\documentclass[a4paper]{article}
\usepackage[english]{babel}

\usepackage{amssymb, amsmath,amsthm}
\usepackage{mathrsfs}
\usepackage{amscd}
\usepackage[active]{srcltx}
\usepackage{verbatim}
\usepackage{graphicx}

\usepackage[backend=bibtex,style=alphabetic]{biblatex}
\usepackage[utf8]{inputenc}
\usepackage[T1]{fontenc}
\usepackage{lmodern}

\usepackage{hyperref} 
\usepackage{todonotes}
\usepackage{enumerate}
\usepackage{mathtools}
\usepackage{bbm}

\mathtoolsset{showonlyrefs,showmanualtags}

\newcommand\dd{\mathrm{d}}
\newcommand\ud{\,\mathrm{d}}

\newcommand{\Aa}{\mathcal{A}}

\newcommand{\Cc}{\mathcal{C}}

\newcommand{\Ii}{\mathcal{I}}

\newcommand{\Oo}{\mathcal{O}}

\newcommand{\Ss}{\mathcal{S}}
\newcommand{\Tt}{\mathcal{T}}

\newcommand{\Xx}{\mathcal{X}}

\newcommand{\fg}{\mathfrak{g}}

\newcommand{\fs}{\mathfrak{s}}

\newcommand{\EE}{\mathbb{E}}

\newcommand{\NN}{\mathbb{N}}

\newcommand{\PP}{\mathbb{P}}

\newcommand{\RR}{\mathbb{R}}

\newcommand\Tr{\mathrm{Tr}}

\newcommand{\inp}[2]{\langle #1,#2 \rangle}

\newcommand{\Nabla}{\nabla}
\newcommand{\Exp}{\mathrm{Exp}}
\newcommand{\ad}{\mathrm{ad}}
\newcommand{\Ad}{\mathrm{Ad}}

\renewcommand{\epsilon}{\varepsilon}

\theoremstyle{plain}
\newtheorem{theorem}{Theorem}[section]
\theoremstyle{remark}

\theoremstyle{plain}

\newtheorem{lemma}[theorem]{Lemma}
\newtheorem{proposition}[theorem]{Proposition}
\newtheorem{definition}[theorem]{Definition}

\numberwithin{equation}{section}

\allowdisplaybreaks

\begin{document}

\title{Large deviations for random walks on Lie groups}

\author{
\renewcommand{\thefootnote}{\arabic{footnote}}
Rik Versendaal\footnotemark[1]
}

\footnotetext[1]{
Delft Institute of Applied Mathematics, Delft University of 
Technology, P.O. Box 5031, 2600 GA Delft, The Netherlands, E-mail: \texttt{R.Versendaal@tudelft.nl}.
}

\date\today

\maketitle

\begin{abstract}
We study large deviations for random walks on Lie groups defined by $\sigma_n^n = \exp(\frac1nX_1)\cdots\exp(\frac1nX_n)$, where $\{X_n\}_{n\geq1}$ is an i.i.d sequence of bounded random variables in the Lie algebra $\fg$. We follow a similar approach as in the proof of large deviations for geodesic random walks as given in \cite{Ver19}. This approach makes it possible to simply rescale the increments of the random walk, without having to resort to dilations in order to reduce the influence of higher order commutators. Finally, we will apply this large deviation result to the Lie group of stochastic matrices. \\

\end{abstract}

\tableofcontents

\section{Introduction}

Since the middle of the previous century, the study of random matrices has gotten a lot of attention. Of particular interest is the limiting behvaviour of products of random matrices. Products of random matrices find their applications for example in the study of wireless telecommunication (see e.g. \cite{TV04}), where a matrix is used to map an input signal to an output signal. The randomness then comes from possible noise disturbing the signal. Another application can be found in studying solutions to difference equations. One can for example think about the Schrodinger equation on a one-dimensional latice with random vector potentionals, see e.g. \cite{BL85}.  

The limiting behaviour of products of random matrices was first studied in \cite{Bel54} and further developed by (among others) \cite{FK60}. In these works, one takes a sequence of matrix valued random variables and studies the product
$$
\Ss_n = M_1\cdots M_n.
$$
In order to say anything about the limiting behaviour of the random variable $\Ss_n$, we take a matrix norm and consider the sequence of real-valued random variables given by $\log||\Ss_n||$. It is then shown that under mild conditions we have 
$$
\lim_{n\to\infty} \frac1n\log||\Ss_n|| = \gamma
$$ 
almost surely, which is the analogue of the law of large numbers. The constant $\gamma$ is referred to as the upper Lyapunov exponent. Furthermore, in \cite{Pag82} (see also \cite{BL85}) it is shown that under additional assumptions, $\log||\Ss_n||$ also satisfies the central limit theorem, in that
$$
\frac{\log||\Ss_n|| - n\gamma}{\sqrt n}
$$
converges in distribution to a Gaussian random variable. Additionally, the same work also verifies the large deviation properties of the sequence $\log||\Ss_nx||$ of random variables, where $x$ is some vector.\\

It is possible to go beyond matrix groups, and study products of elements of a general Lie group. For a sequence $g_1,g_2,\ldots$ in a Lie group $G$, using the group operation, we can define the product
$$
\Ss_n = g_1g_2\cdots g_n,
$$
and we will refer to this as a random walk in the Lie group G.

Now, in order to study limit theorems like the law of large numbers and central limit theorem, we can no longer use a matrix norm. Instead, we can equip $G$ with a left-invariant Riemannian metric $d$ and study the real-valued random variables $d(\Ss_n,e)$, where $e$ is the identity element of the group $G$. It is shown in \cite{Gui80} that if $G$ is locally compact, then there exists a $\gamma \geq 0$ such that almost surely 
$$
\lim_{n\to\infty} \frac1n d(\Ss_n,e) = \gamma.
$$

Furthemore, the central limit theorem, i.e., the convergence of
$$
\frac{d(\Ss_n,e) - \gamma n}{\sqrt{n}}
$$
in distribution to a normal distribution is studied in \cite{Tut65}.\\

Another approach to study limit theorems, which we will be considering here, is not to transfer the problem to a real-valued setting, but to find a suitable way of rescaling the random walk in the Lie group $G$ itself. For this, we slightly modify the definition of a random walk. Let $\fg$ denote the Lie algebra of $G$, and let $X_1,X_2,\ldots$ be a sequence in $\fg$. We then define the random walk in $G$ as
$$
\Ss_n = \exp(X_1)\cdots\exp(X_n),
$$ 
where $\exp:\fg \to G$ denotes the exponential map. Because $\fg$ is a vector space, we can rescale the sequence $X_1,X_2,\ldots$, allowing us to define the rescaled random walk by
$$
\sigma_n^n = \exp\left(\frac1nX_1\right)\cdots\exp\left(\frac1nX_n\right). 
$$
However, from the Baker-Campbell-Hausdorff formula it follows after a formal computation that
$$
\sigma_n^n = \exp\left(\frac1n\sum_{i=1}^n X_i + \Oo(1)\right),
$$
which one obtains by counting the number of commutators.
Consequently, it is not obvious how to use known results regarding the limiting behaviour of $\frac1n\sum_{i=1}^nX_i$ in order to study the limiting behaviour of $\sigma_n^n$. To overcome this problem, instead of simply rescaling the elements $\fg$ by $\frac1n$, one uses so called dilations $D_{\frac1n}:\fg \to \fg$ as done in \cite{Bre04,BC99,Gav77,Neu96}. The idea is to decompose an element $Y \in \fg$ as $Y = \sum_{i\geq1} Y_i$, where $Y_i$ is an $i$-th order commutator, meaning it is of the form $[Y_i^1,[\cdots,[Y_i^{i-1},Y_i^i]]]$, where none  of the $Y_i^j$ are commutators. We call a Lie algebra nilpotent, if there is some $l \in \NN$ such that all comutators of order $l$ are 0. In that case, $Y$ may be written as a finite sum $Y = \sum_{i=1}^l Y_i$ and we define the dilation $D_{\frac1n}Y$ of $Y$ by
$$
D_{\frac1n}Y = \sum_{i=1}^l \frac{1}{n^i}Y_i
$$
So essentially, we dilate the elements of $\fg$ in such a way that the problematic parts, being the (higher order) commutators, are scaled away in the limit by multiplying those by higher powers of $\frac1n$. Now the Baker-Campbell-Hausdorff formula will give us after a formal computation that
$$
\prod_{i=1}^n \exp\left(D_{\frac1n}X_i\right) = \exp\left(\frac1n\sum_{i=1}^n X_i + \Oo\left(\frac1n\right)\right),
$$
making it at least more viable that in the limit this product should indeed behave like $\exp\left(\frac1n\sum_{i=1}^n X_i\right)$. It is shown in \cite{Gav77,Neu96} that the law of large numbers is satisfied, i.e., if $X_1,X_2,\ldots$ are i.i.d. with $\EE(X_1) = 0$ and with finite moment generating function in a neighbourhood of the origin, then almost surely
$$
\lim_{n\to\infty} \prod_{i=1}^n \exp\left(D_{\frac1n}X_i\right) = 0.
$$
The large deviations for the sequence 
\begin{equation}\label{eq:sequence_scaled_away}
\left\{\prod_{i=1}^n \exp\left(D_{\frac1n}X_i\right)\right\}_{n\geq0}
\end{equation}
are studied in \cite{BC99}. The prove uses path-space, first transferring the problem to $\RR^d$ to use Mogulskii's theorem, followed up by the contraction principle to get the large deviations for the end-point of the random walk.\\ 

However, if $G$ admits a bi-invariant metric, the processes $\Ss_n$ and $\sigma_n^n$ are special cases of geodesic random walks as defined in \cite{Jor75}. The large deviations for these have been studied in \cite{Ver19,KRV18}. Consequently, if $G$ admits a bi-invariant metric, then the sequence $\{\sigma_n^n\}_{n\geq1}$ satisfies in $G$ the large deviation principle. Moreover, the corresponding rate function coincides with the rate function for the sequence of random variables in \eqref{eq:sequence_scaled_away}, where the higher order commutators are scaled away.

This raises the question whether the sequence $\{\sigma_n^n\}_{n\geq1}$ also satisfies a large deviation principle when $G$ does not necessarly admit a bi-invariant metric. Following the approach in \cite{Ver19}, we will show that under some assumptions, this is indeed the case. More precisely, we will prove that if $\{X_n\}_{n\geq1}$ is a sequence of bounded, i.i.d. $\fg$-valued random variables, with $\EE(X_1) = 0$ and everywhere finite moment generating function, then the sequence $\{\sigma_n^n\}_{n\geq0}$ satisfies in $G$ the large deviation principle with rate function $I$ given by
$$
I(g) = \inf\left\{\int_0^1 \Lambda^*(\dot\gamma(t))\ud t \middle| \gamma \in \Aa\Cc([0,1];G), \gamma(0) = e, \gamma(1) = g\right\}.
$$
Here, $\Lambda(\lambda) = \log\EE(e^{\inp{\lambda}{X_1}})$ denote the log moment generating function, and $\Lambda^*$ its Legendre transform given by
$$
\Lambda^*(X) := \sup_{\lambda \in \fg} \inp{\lambda}{X} - \Lambda(\lambda).
$$
\\

The paper is organised as follows. First, in Section \ref{section:notation} we make precise the notion of a large deviation principle for a sequence of random variables. Additionally, we introduce some theory on Lie groups on Lie algebras and fix the notation we use in what follows. With the notation fixed, we define in Section \ref{section:random_walks} the random walks in Lie groups we will be studying. In Section \ref{section:main_theorem} we state our main theorem and give a sketch of its proof. Additionally, we will also discuss an example by considering the stochastic group. Section \ref{section:estimates} is devoted to important estimates following from the Baker-Campbell-Hausdorff formula. Finally, we use these estimates to prove our main theorem in Section \ref{section:proof}.


\section{Notation and basic theory} \label{section:notation}

In this section we introduce some large deviation theory as well as the theory of Lie groups and Lie algebras. Additionally, we will fix the notation we will use in what follows.

\subsection{Large deviations}

Large deviation theory is concerned with the limiting behaviour on an exponential scale of a sequence $\{Z_n\}_{n\geq1}$ of random variables. This behaviour is determined by a so called rate function. More precisely, we have the following definition.

\begin{definition}
Let $\{Z_n\}_{n\geq1}$ be a sequence of random variables taking values in some metric space $\Xx$. 

\begin{enumerate}
\item A \emph{rate function} is a lower semicontinuous function $I:\Xx \to [0,\infty]$. If the level sets of $I$ are compact, then the rate function is called \emph{good}.
\item The sequence $\{Z_n\}_{n\geq1}$ satisfies the large deviation principle in $\Xx$ with rate function $I$ if the following hold:
\begin{enumerate}
\item (Upper bound) For any $F \subset \Xx$ closed we have
$$
\limsup_{n\to\infty} \frac1n\log\PP(Z_n \in F) \leq -\inf_{x \in F} I(x).
$$
\item (Lower bound) For any $G \subset \Xx$ open we have
$$
\liminf_{n\to\infty} \frac1n\log\PP(Z_n \in G) \geq -\inf_{x\in G} I(x).
$$
\end{enumerate}
\end{enumerate}
\end{definition}

It is often easier to prove the upper bound for compact sets, rather than general closed sets. If the lower bound holds, and the upper bound only holds for compact sets, we say the sequence $\{Z_n\}_{n\geq1}$ satisfies the \emph{weak large deviation principle}. If the mass of the random variables is then concentrated enough on compact sets, then the upper bound may actually be extended to all closed sets. More precisely, we say that the sequence $\{Z_n\}_{n\geq1}$ is \emph{exponentially tight} if for every $\alpha > 0$ there exists a compact set $K_\alpha \subset \Xx$ such that
$$
\limsup_{n\to\infty} \frac1n\log\PP(Z_n \in K_\alpha^c) < - \alpha.
$$

We have the following proposition, which can for example be found in \cite[Section 1.2]{DZ98}.

\begin{proposition}
Let $\{Z_n\}_{n\geq1}$ be a sequence of random variables satisfying the weak large deviation principle with rate function $I$. Assume furthermore that the sequence is exponentially tight. Then $\{Z_n\}_{n\geq1}$ satisfies the (full) large deviation principle with the same rate function $I$.
\end{proposition}

\subsection{Lie groups and Lie algebras} \label{section:Lie_groups}

In this section we collect the necessary notation and theory on Lie groups and Lie algebras. For more details, we refer to \cite{Lee03,War83}  for general Lie group theory, and to \cite{Hal15} for a treatment of matrix Lie groups. \\

Let $G$ be a finite-dimensional Lie group, i.e., a finite dimensional group with a smooth manifold structure such that the group operations of multiplication and inversion are smooth. We write $e$ for the identity element of $G$. The Lie algebra $\fg$ of $G$ is defined as the tangent space $T_eG$ at the identity.\\

Next, we want to equip $\fg$ with a \emph{Lie bracket} $[\cdot,\cdot]$, which is a map from $\fg\times\fg$ into $\fg$ which is bilinear, skew-symmetric and satisfies the Jacobi identity:
$$
[X,[Y,Z]] + [Y,[Z,X]] + [Z,[X,Y]] = 0,
$$
for all $X,Y,Z \in \fg$. In order to construct such a Lie bracket, we need a different interpretation of the Lie algebra $\fg$. 

To this end, we denote by $L_g:G \to G$ left multiplication with $g$. A vector field $V$ on $G$ is called \textit{left-invariant} if for all $g,h \in G$ we have $\dd L_g(h)(V(h)) = V(gh)$.

For every $X \in \fg$, we can define a left-invariant vector field $X^L$ on $G$ by setting 
\begin{equation}\label{eq:left_invariant_vector_field}
X^L(g) = \dd L_g(e)(X).
\end{equation}
This is a vector space isomorphism, with inverse given by the evaluation of the vector field at the identity $e$. Consequently, the Lie algebra $\fg$ of $G$ may be identified with the set of left-invariant vector fields on $G$. This set forms a Lie algebra under the Lie bracket $[V,W] = VW - WV$. Consequently, we define the Lie bracket $[X,Y]$ for $X,Y \in \fg$ by $[X,Y] := [X^L,Y^L](e)$. \\

The above prodecure also shows us that for every $g \in G$ we can identify the tangent space $T_gM$ with $\fg$ via the isomorphism $\dd L_g(e):\fg \to T_gM$. Whenever we consider a tangent vector $X \in T_gM$ as elements of $\fg$, we have this identification in mind.

\subsubsection{Exponential map}

We now define an important function that allows us to map elements of the Lie algebra to the Lie group. For every $X \in \fg$, there exists a curve $\gamma_X:\RR \to G$ satisfying $\gamma_X(0) = e$ and $\dot\gamma_X(t) = X^L(\gamma_X(t))$ (note that $\dot\gamma(0) = X$ in this case). Using this curve, we define the \textit{exponential map} $\exp:\fg \to G$ by $\exp(X) = \gamma_X(1)$. \\

For every $X \in \fg$ we have 
$$
\dd\exp(0)(X) = \frac{\dd}{\dd t}\Big|_{t = 0} \exp(tX) = X
$$
so that $\dd\exp(0) = I$. Consequently, by the inverse function theorem, there exists an $r > 0$ such that $\exp$ is homeomorphism from $B(0,r)$ onto its image. The inverse of the exponential map is refered to as the \emph{logarithm map}, and is denoted by $\log$. We have the following proposition.

\begin{proposition}\label{prop:injectivity}
For every $r > 0$ such that $\exp$ is a homeomorphism on $\overline{B(0,r)}$, there exists an $\epsilon > 0$ such that $\log$ is well-defined on $\overline{B(e,\epsilon)}$ and such that for all $g \in \overline{B(e,\epsilon)}$ we have $|\log(g)| \leq r$.
\end{proposition}
\begin{proof}
Because $\exp$ is a homeomorphism, it is an open map, and hence $\exp(\overline{B(0,r)})$ contains some open ball $B(e,\epsilon)$. Because $\exp(\overline{B(0,r)})$ is closed, it must be that $\overline{B(e,\epsilon)} \subset \exp(\overline{B(0,r)})$ so that $\log$ is well defined on $\overline{B(e,\epsilon)}$ and $\log(\overline{B(e,\epsilon)}) \subset \overline{B(0,r)}$ as desired.
\end{proof}

\subsubsection{Riemannian metric}

For reasons that will become apparent later, we equip $\fg$ with an inner product $\inp{\cdot}{\cdot}$. This induces on $\fg$ a norm $|\cdot|$ given by $|X| = \sqrt{\inp{X}{X}}$. Because $\fg$ is finite-dimensional, all norms are equivalent, and hence, our results will not depend on the choice of inner product. 

The inner product on $\fg$ may be extended to a Riemannian metric on $G$. For this, we use the fact that $T_gM$ may be identified with $\fg$ via the isomorphism $\dd L_g(e)$. With this identification in mind, we can define an inner product on $T_gG$ by 
$$
\inp{X}{Y}_g = \inp{\dd L_g(e)^{-1}X}{\dd L_g(e)^{-1}Y},
$$
The assumption that the group operations are smooth implies that this defines a Riemannian metric on $G$. By construction this Riemannian metric is left-invariant, i.e., for all $g,h \in G$ and for all $X,Y \in T_gG$ we have
$$
\inp{\dd L_h(g)X}{\dd L_h(g)Y}_{hg} = \inp{X}{Y}_g.
$$
This shows that $\dd L_h(g):T_gG \to T_hG$ is an isometry. In particular, the identification $\dd L_g(e):\fg \to T_gG$ of $T_gG$ with the Lie algebra $\fg$ is also an isometry. Consequently, if we consider $X \in T_gG$ as element in $\fg$, its norm can also be taken as element of $\fg$\\. 

To the Riemannian metric we can associate a Riemannian distance $d:G\times G \to \RR$ given by
$$
d(g,h) = \inf\left\{\int_0^1 |\dot\gamma(t)| \ud t \middle | \gamma:[0,1] \to G \mbox{ piecewise smooth}, \gamma(0) = g, \gamma(1) = h\right\}.
$$
Because the Riemannian metric is left-invariant, it follows that for all $f,g,h \in G$ we have
\begin{equation}\label{eq:left_invariance_distance}
d(g,h) = d(fg,fh).
\end{equation}
This shows that the distance between elements of $G$ is preserved under left-multiplication.


\section{Random walks in Lie groups}\label{section:random_walks}

In this section we introduce the concept of a random walk in a general (connected) Lie group $G$. We will relate this concept to geodesic random walks in general Riemannian manifolds, as defined in \cite{Jor75,Ver19,KRV18}.

\subsection{Definition of a random walk in $G$.}

Let $\mu$ be a measure on the Lie algebra $\fg$. Let $X_1,X_2,\ldots$ be a sequence of i.i.d. random variables in $\fg$ with distribution $\mu$. We define the random walk $\Ss_n \in G$ by 
\begin{equation}\label{eq:random_walk}
\Ss_n = \exp(X_1)\exp(X_2)\cdots\exp(X_n).
\end{equation}
Furthermore, we define the rescaled random walk by
\begin{equation}\label{eq:rescaled_random_walk}
\sigma_n^n = \Exp\left(\frac1nX_1\right)\Exp\left(\frac1nX_2\right)\cdots\Exp\left(\frac1nX_n\right).
\end{equation}

In order to relate this to the concept of geodesic random walks in \cite{Jor75,Ver19,KRV18}, we need to argue how one-parameter subgroups of the form $\gamma(t) = g\exp(tX)$ can be interpreted as geodesics. To this end, we need some additional theory from Lie groups.

\begin{definition}
Let $\Nabla$ be a connection on a Lie group $G$. $\Nabla$ is said to be left-invariant if for any two left-invariant vector fields $X^L$ and $Y^L$ (see \eqref{eq:left_invariant_vector_field}) with $X,Y \in \fg$ we have that $\Nabla_{X^L}Y^L$ is also left-invariant.
\end{definition}

Among the left-invariant connections, there are special connections for which the one-parameter subgroups form geodesics.

\begin{definition}
A Cartan connection on a Lie group $G$ is a left-invariant connection satisfying the property that the subgroup $\gamma(t) = \exp(tX)$ is a geodesic for every $X \in \fg$.
\end{definition}

One question that arises, is whether such connections always exist. This is indeed the case, as the following result from \cite{Mil76} states.

\begin{proposition}\label{prop:Cartan_connection}
For any Lie group $G$ there exists a unique symmetric Cartan connection $\Nabla$ given by
$$
\Nabla_{X^L}Y^L = \frac12[X,Y]^L
$$
for any $X,Y \in \fg$. 
\end{proposition}

Consequently, our definition of a random walk on $G$ coincides with the definition of a geodesic random walk when we equip $G$ with a Cartan connection. Although this connection can be chosen to be symmetric, it is in general not possible to choose it so that it is also compatible with the Riemannian metric. In the case of a bi-invariant metric, this is possible, and the Cartan connection in Proposition \ref{prop:Cartan_connection} is compatible with the metric and thus coincides with the Levi-Civita connection. In this case, the exponential map $\exp:\fg \to G$ coincides with the Riemannian exponential map. 

In order to connect our result to the results in \cite{Ver19,KRV18}, we need to show that our measure $\mu$ can be extended to a collection of measures $\{\mu_g\}_{g\in G}$ which are invariant under parallel transport in the sense of \cite[Definition 3.7]{Ver19}. Along geodesics of the form $\gamma(t) = \exp(tX)g$, parallel transport is given by $\dd L_{\exp(tX)}(g)$. If we now set 
$$
\mu_g = \mu \circ \dd L_g(e)^{-1},
$$
then the bi-invarience of the Riemannian metric shows that the collection $\{\mu_g\}_{g\in G}$ is invariant under parallel transport, at least when transporting along geodesics. It is actually possible to show that in this case, the invariance also holds along arbitrary piecewise smooth curves.

The above shows that if $G$ admits a bi-invariant metric, the our notion of a random walk with i.i.d. increments coincides with the notion of a geodesic random walk with i.i.d. increments as in \cite{Ver19,KRV18}. Although the proof is somewhat simpler, our results do not add anything over the results in \cite{Ver19,KRV18} if $G$ admits a bi-invariant metric. The novelty is in the case when no such bi-invariant metric exists.


\section{Main theorem, sketch of the proof and an example}\label{section:main_theorem}

With all the notation fixed, we are ready to state in this section the main theorem that we are going to prove. Because the proof consists of a number of steps, we also provide a sketch of the proof so that the main steps are clear. The precise proof will be given in Section \ref{section:proof}. We conclude the section by showing how the main theorem can be applied if we consider the Lie group of stochastic matrices.

\subsection{Statement of the main theorem}

Let $G$ be a Lie group with Lie algebra $\fg$ equiped with an inner product $\inp{\cdot}{\cdot}$. Let $\{X_n\}_{n\geq1}$ be a sequence of i.i.d. random variables in the Lie algebra $\fg$ and denote by $\sigma_n^n$ the rescaled random walk as in \eqref{eq:rescaled_random_walk}. We are going to prove that under some assumptions on the increments $\{X_n\}_{n\geq1}$, the sequence $\{\sigma_n^n\}_{n\geq1}$ satisfies a large deviation principle in $G$.

Along with the large deviation principle for $\{\sigma_n^n\}_{n\geq1}$, we need to identify the corresponding rate function. If $G$ admits a bi-invariant metric, it follows from \cite[Theorem 4.1]{Ver19} or \cite[Theorem 5.5]{KRV18} that the rate function is given by
$$
I(g) = \inf\{\Lambda^*(X)|\exp(X) = g\}.
$$
Here, $\Lambda(\lambda)$ is the log moment generating function of an increment, given by
$$
\Lambda(\lambda) := \log\EE\left(e^{\inp{\lambda}{X_1}}\right),
$$ 
while $\Lambda^*$ denotes its Legendre transform, defined as
$$
\Lambda^*(X) := \sup_{\lambda \in \fg}  \inp{\lambda}{X} - \Lambda(\lambda).
$$

Obtaining this form of the rate function relies on the fact that if we minimize $\int_0^1 \Lambda^*(\dot\gamma(t)) \ud t$ over curves with fixed endpoints, the minimum is attained by a geodesic. However, if $G$ does not admit a bi-invariant metric, curves of the form $\gamma(t) = \exp(tX)$ are no longer necessarily geodesics (when taking the exponential map in the terminology of Lie groups and Lie algebra's). Consequently, we can do no better than the expression
$$
I(g) = \inf\left\{\int_0^1 \Lambda^*(\dot\gamma(t))\ud t \middle | \gamma \in \Aa\Cc([0,1];G), \gamma(0) = e, \gamma(1) = g\right\}.
$$

We now collect everything and give the statement of the theorem.

\begin{theorem}\label{theorem:Cramer_Lie}
Let $G$ be a Lie group and $\fg$ its associated Lie algebra, equipped with an inner product $\inp{\cdot}{\cdot}$. Let $\{X_n\}_{n\geq1}$ be a sequence of random variables in $\fg$ and denote by $\sigma_n^n$ the associated rescaled random walk as in \eqref{eq:rescaled_random_walk}. Assume that the sequence $\{X_n\}_{n\geq1}$ is i.i.d. and bounded, and assume that the log moment generating function $\Lambda(\lambda) = \log\EE(e^{\inp{\lambda}{X_1}})$ is everywhere finite. Then the sequence $\{\sigma_n^n\}_{n\geq0}$ satisfies in $G$ the large deviation principle with good rate function
\begin{equation}\label{eq:ratefunction}
I_G(g) = \inf\left\{\int_0^1 \Lambda^*(\dot\gamma(t))\ud t \middle | \gamma \in \Aa\Cc([0,1];G), \gamma(0) = e, \gamma(1) = g\right\}.
\end{equation}
\end{theorem}

Because the proof of Theorem \ref{theorem:Cramer_Lie} is rather long, we first provide a sketch of the proof, before we get to the actual details in Section \ref{section:proof}.

\subsection{Sketch of the proof of Theorem \ref{theorem:Cramer_Lie}}\label{section:sketch}

The proof of Theorem \ref{theorem:Cramer_Lie} is inspired by the proof of \cite[Theorem 4.1]{Ver19}, and consequently, we will follow similar steps as explained in \cite[Section 4]{Ver19}. Like in the proof of the large deviations for rescaled random walks in Euclidean space, known as Cram\'er's theorem, we prove the upper and lower bound seperately.\\

By Cram\'er's theorem for vector spaces (see e.g. \cite[Chapter 2]{DZ98} or \cite[Chapter 1]{Hol00}), the sequence $\{\frac1n\sum_{i=1}^n X_i\}_{n\geq1}$ of empirical averages satisfies the large deviation principle in $\fg$ with good rate function $I(X) = \Lambda^*(X)$. Consequently, by the contraction principle (see e.g. \cite[Chapter 4]{DZ98}), the sequence $\{\Sigma_n\}_{n\geq 1}$ given by
$$
\Sigma_n = \exp\left(\frac1n\sum_{i=1}^n X_i\right)
$$
satisfies the large deviation principle in $G$ with good rate function
$$
I_G(g) = \inf\{\Lambda^*(X)|\exp(X) = g\}.
$$

Unfortunately, the Baker-Campbell-Hausdorff formula shows us that in general, $\Sigma_n$ and $\sigma_n^n$ do not coincide. More precisely, given that the random walk stays close enough to the identity $e$, so that logarithms are well-defined, the integral version of the Baker-Campbell-Hausdorff formula (see Theorem \ref{theorem:BCH}) gives us that
\begin{equation}\label{eq:Taylor_sketch}
\log(\sigma_n^n) = \frac1n\sum_{i=1}^n \left(\int_0^1 \frac{\ad_{\log(\sigma_{i-1}^n)}}{1 - e^{-\ad_{\log(\sigma_{i-1}^n)}}}\right)X_i
\end{equation}
Here, $\sigma_i^n$ is defined to be the point of the random walk after $i$ steps, i.e.,
$$
\sigma_i^n = \exp\left(\frac1nX_1\right)\cdots\exp\left(\frac1nX_i\right).
$$




However, we would like to understand the difference between $\log(\sigma_n^n)$ tand $\frac1n\sum_{i=1}^n X_i$. For this, we compare $\frac1n\sum_{i=1}^n X_i$ to the expression found in \eqref{eq:Taylor_sketch} for $\log(\sigma_n^n)$. We prove (see Proposition \ref{prop:log_product}) that there exists constants $C_{|\log(\sigma_{i-1}^n)|}$ such that
$$
\left|\left(\int_0^1 \frac{\ad_{\log(\sigma_{i-1}^n)}}{1 - e^{-\ad_{\log(\sigma_{i-1}^n)}}}\right) X_i - X_i\right| \leq C_{|\log(\sigma_{i-1}^n)|}|X_i|,
$$
where $C_\alpha$ is a constant, decreasing in $\alpha$ and such that $\lim_{\alpha \to 0} C_\alpha = 0$.

Using the triangle inequality and the smoothness of $\log$, one can show that $|\log(\sigma_i^n)| \lesssim \frac in B$, where $B$ is the uniform bound on the increments. Consequently, $C_{|\log(\sigma_{i-1}^n)|} \leq C_B$ for all $i = 1,\ldots,n$. If we now collect everything, we find
\begin{equation}\label{eq:comparison_sketch}
\left|\log(\sigma_n^n) - \frac1n\sum_{i=1}^n X_i\right| \leq C_BB.
\end{equation}

Because $B$ is fixed, this upper bound unfortunately does not show us that that $\log(\sigma_n^n)$ and $\frac1n\sum_{i=1}^n X_i$ will get arbitrarily close if $n$ tends to infinity. The key will be to decrease the constant $C_BB$ in an appropriate way.\\ 

To do this, we split the random walk into finitely many, say $m$, pieces, each consisting of $\lfloor m^{-1}n\rfloor$ increments. It turns out that this also takes care of the problem that the logarithms we use are not necessarily well-defined. More precisely, for $m \in \NN$ we define the indices $n_j = j\lfloor m^{-1}n\rfloor$ for $j = 0,\ldots,m-1$ and set $n_m = n$. We can prove (see \eqref{eq:estimate_distance} and \eqref{eq:distance_pieces}) that if $B$ is the uniform bound on the increments, then for every $j = 1,\ldots,m$ and $i = 1,\ldots,n_j - n_{j-1}$ we have
$$
d(e,(\sigma_{n_{j-1}}^n)^{-1}\sigma_{n_{j-1}+i}^n) = d(\sigma_{n_{j-1}}^n,\sigma_{n_{j-1}+i}^n) \leq \frac{i}{n}B \leq \frac1mB.
$$
Here, the first equality follows from the left-invariance of the metric $d$. Consequently, if $m \in \NN$ is large enough, then $\log((\sigma_{n_{j-1}}^n)^{-1}\sigma_{n_{j-1}+i}^n)$ is well-defined for every $j = 1,\ldots,m$ and every $i = 1,\ldots,n_j-n_{j-1}$. In particular, one may show in a similar spirit as \eqref{eq:comparison_sketch}, that
\begin{equation}\label{eq:replacement2}
\left|\log((\sigma_{n_{j-1}}^n)^{-1}\sigma_{n_j}^n) - \frac1n\sum_{i=1}^{n_j - n_{j-1}} X_{n_{j-1}+i}\right| \leq C_{m^{-1}B}\frac Bm.
\end{equation}

Now let us define $Y_{\lfloor m^{-1}n\rfloor}^{n,m,j} = \log((\sigma_{n_{j-1}}^n)^{-1}\sigma_{n_j}^n) \in \fg$ for $j = 1,\ldots,m$. By the above construction, we have that
$$
\sigma_n^n = \exp\left(Y_{\lfloor m^{-1}n\rfloor}^{n,m,1}\right)\cdots\exp\left(Y_{\lfloor m^{-1}n\rfloor}^{n,m,m}\right) =: \Psi_m\left(Y_{\lfloor m^{-1}n\rfloor}^{n,m,1},\ldots,Y_{\lfloor m^{-1}n\rfloor}^{n,m,m}\right),
$$  
where $\Psi_m:\fg^m \to G$ is the continuous function given by
$$
\Psi_m(x_1,\ldots,x_m) = \exp(x_1)\cdots\exp(x_m).
$$
Using this, we prove the upper and lower bound for the large deviation principle for $\{\sigma_n^n\}_{n\geq1}$, which we explain in the upcoming two sections.

\subsubsection{Upper bound of the large deviation principle for $\{\sigma_n^n\}_{n\geq1}$.}\label{section:sketch_upper}

In this section we sketch the proof of the upper bound of the large deviation principle for $\{\sigma_n^n\}_{n\geq1}$. For $F \subset G$ closed, and every $m \in \NN$ large enough, we have that $\Psi_m^{-1}F \subset \fg^m$ is closed and
\begin{equation}\label{eq:repeated_exponential_sketch}
\PP(\sigma_n^n \in F) \leq \PP\left(\left(Y_{\lfloor m^{-1}n\rfloor}^{n,m,1},\ldots,Y_{\lfloor m^{-1}n\rfloor}^{n,m,m}\right) \in \Psi_m^{-1}F\right).
\end{equation}

Because $\fg^m$ is a vector space, we can use a similar argument as in the proof of Cram\'er's theorem for the Euclidean setting (see e.g. \cite{DZ98,Hol00}), to obtain that
\begin{align*}
&\limsup_{n\to\infty} \frac1n\log\PP(\sigma_n^n \in F)
\\
&\leq
\limsup_{n\to\infty} \frac1n\log\PP\left(\left(Y_{\lfloor m^{-1}n\rfloor}^{n,m,1},\ldots,Y_{\lfloor m^{-1}n\rfloor}^{n,m,m}\right) \in \Psi_m^{-1}F\right)
\\
&\leq
-\inf_{x \in \fg^m} \sup_{\lambda \in \fg^m} \inp{\lambda}{x} - \limsup_{n\to\infty} \frac1n\log\EE\left(e^{n\inp{\lambda}{(Y_{\lfloor m^{-1}n\rfloor}^{n,m,1},\ldots,Y_{\lfloor m^{-1}n\rfloor}^{n,m,m})}}\right) 
\end{align*}

Now one can use \eqref{eq:replacement2} to prove that
$$
\EE\left(e^{n\inp{\lambda}{(Y_{\lfloor m^{-1}n\rfloor}^{n,m,1},\ldots,Y_{\lfloor m^{-1}n\rfloor}^{n,m,m})}}\right)  \leq e^{C_{m^{-1}B}|\lambda|Bm^{-1}}\EE\left(e^{n\inp{\lambda}{(Z_1^{m,n},\ldots,Z_m^{n,m})}}\right), 
$$
where
$$
Z_j^{m,n} = \frac1n\sum_{i=1}^{n_j - n_{j-1}} X_{n_{j-1} + i}.
$$
Now, because the sequence $\{X_n\}_{n\geq1}$ is i.i.d., the random variables $Z_1^{n,m},\ldots,Z_m^{n,m}$ are also i.i.d. with $\EE(e^{n\inp{\lambda}{Z_1^{n,m}}}) = \EE(e^{\inp{\lambda}{X_1}})^{\lfloor m^{-1}n\rfloor}$. Consequently, we find that
$$
\EE\left(e^{\inp{\lambda}{(Z_1^{m,n},\ldots,Z_m^{n,m})}}\right) = M(\lambda_1)^{\lfloor m^{-1}n\rfloor}\cdots M(\lambda_m)^{\lfloor m^{-1}n\rfloor}.
$$

Collecting everything, we find that
\begin{align*}
&\limsup_{n\to\infty} \frac1n\log\PP(\sigma_n^n \in F)
\\
&\leq
-\inf_{x \in \fg^m} \sup_{\lambda \in \fg^m} \left\{ \inp{\lambda}{x} - \frac1m\sum_{j=1}^m \Lambda(\lambda_i) - C_{m^{-1}B}|\lambda|B\frac1m\right\}
\\
&=
-\inf_{x \in \fg^m} \frac1m\sum_{j=1}^m \sup_{\lambda \in \fg} \left\{\inp{\lambda}{mx_j} - \Lambda(\lambda) - C_{m^{-1}B}|\lambda|B\right\}.
\end{align*}

Finally, by letting $m$ tend to infinity, apart from some technical difficulties, one obtains
$$
\limsup_{n\to\infty} \frac1n\log\PP(\sigma_n^n \in F) \leq -\inf_{g \in G} I_G(g),
$$
as desired.

\subsubsection{Lower bound of the large deviation principle for $\{\sigma_n^n\}_{n\geq1}$.}

To prove the lower bound of the large deviation principle for $\{\sigma_n^n\}_{n\geq1}$, we first observe that it is sufficient to show for every $U \subset G$ open and every $g \in G$ that
$$
\liminf_{n\to\infty} \frac1n\log\PP(\sigma_n^n \in G) \geq -\int_0^1 \Lambda^*(\dot\gamma(t))\ud dt 
$$
for all $\gamma \in \Aa\Cc([0,1];G)$ with $\gamma(0) = e$ and $\gamma(1) = g$.

To do this, we fix $\gamma \in \Aa\Cc([0,1];G)$ with $\gamma(0) = e$ and $\gamma(1) = g$ and define for $m \in \NN$ the vectors
$$
y_i^m := \log\left(\gamma\left(\frac{i-1}{m}\right)^{-1}\gamma\left(\frac im\right)\right) \in \fg.
$$
Note that $\Psi_m((y_1^m,\ldots,y_m^m)) = g$, where $\Psi_m:\fg^m \to G$ is as in \eqref{eq:repeated_exponential_sketch}. In order to continue, we need to know a bit more about the continuity properties of $\Psi_m$. More precisely, we will prove (see Proposition \ref{prop:continuity_repeated_exponential}) that there exists a constant $C > 0$ such that for $\epsilon > 0$ and $m \in \NN$ large enough, we have that if
$$
(x_1,\ldots,x_m) \in \prod_{i=1}^m B(y_i^m,(Cm)^{-1}\epsilon),
$$
then
$$ 
\Psi_m((x_1,\ldots,x_m)) \in B(\Psi_m((y_1^m,\ldots,y_m^m)),\epsilon) = B(g,\epsilon).
$$

Now, because the fundamental theorem of calculus fails, in that (compare to the Euclidean case)
$$
\log\left(\gamma\left(\frac{i-1}{m}\right)^{-1}\gamma\left(\frac im\right)\right) \neq \int_{\frac{i-1}{m}}^{\frac im} \dot\gamma(t) \ud t.
$$
We will show that under the condition that $\dot\gamma$ is bounded (see Proposition \ref{prop:fundamental_theorem}), we have
$$
\left|\log\left(\gamma\left(\frac{i-1}{m}\right)^{-1}\gamma\left(\frac im\right)\right) - \int_{\frac{i-1}{m}}^{\frac im} \dot\gamma(t) \ud t\right| \leq L_m\frac1m,
$$
where $\lim_{m\to\infty} L_m = \infty$. In particular, if we set
$$
\tilde y_i^m := \int_{\frac{i-1}{m}}^{\frac im} \dot\gamma(t) \ud t,
$$
then for $m$ large enough we have $B(\tilde y_i^m,(2Cm)^{-1}\epsilon) \subset B(y_i^m,(Cm)^{-1}\epsilon)$. Consequently, we have that if 
$$
(x_1,\ldots,x_m) \in \prod_{i=1}^m B(\tilde y_i^m,(2Cm)^{-1}\epsilon),
$$
then $\Psi_m((x_1,\ldots,x_m)) \in  B(g,\epsilon)$.\\

Because $U$ is open, there exists an $\epsilon > 0$ such that $B(\epsilon,g) \subset U$. Using the above continuity property, we find that
\begin{align*}
&\PP(\sigma_n^n \in U)
\\
&\geq 
\PP\left(\left(Y_{\lfloor m^{-1}n\rfloor}^{n,m,1},\ldots,Y_{\lfloor m^{-1}n\rfloor}^{n,m,m}\right) \in B(\tilde y_1^m,(2Cm)^{-1}\epsilon) \times \cdots \times B(\tilde y_m^m,(2Cm)^{-1}\epsilon)\right)
\end{align*}

Now using \eqref{eq:replacement2} and the fact that $\lim_{m\to\infty} C_{m^{-1}B} = 0$, we have for $m$ large enough that
$$
\left|Y_{\lfloor m^{-1}n\rfloor}^{n,m,j} - \frac1n\sum_{i=1}^{n_j - n_{j-1}}X_{n_{j-1} + i}\right| \leq (2Cm)^{-1}\frac\epsilon2.
$$
But then we find that
\begin{align*}
&\PP\left(\left(Y_{\lfloor m^{-1}n\rfloor}^{n,m,1},\ldots,Y_{\lfloor m^{-1}n\rfloor}^{n,m,m}\right) \in B(\tilde y_1^m,(2Cm)^{-1}\epsilon) \times \cdots \times B(\tilde y_m^m,(2Cm)^{-1}\epsilon)\right)
\\
&\geq
\PP\left(\left(\frac1n\sum_{i=1}^{n_1}X_i,\ldots,\frac1n\sum_{i=1}^{n_m - n_{m-1}}X_{n_{m-1} + i}\right) \in B(\tilde y_1^m,(2Cm)^{-1}\epsilon/2) \times \cdots \times B(\tilde y_m^m,(2Cm)^{-1}\epsilon/2)\right)
\\
&=
\prod_{j=1}^m \PP\left(\frac1n\sum_{i=1}^{n_j - n_{j-1}}X_{n_{j-1} + i} \in B(\tilde y_j^m,(2Cm)^{-1}\epsilon/2) \right).
\end{align*}

By Cram\'er's theorem for random walks in Euclidean space, we find that
$$
\liminf_{n\to\infty} \frac1n\log\PP\left(\frac1n\sum_{i=1}^{n_j - n_{j-1}}X_{n_{j-1} + i} \in B(y_j,(Cm)^{-1}\epsilon) \right) \geq -\frac1m\Lambda^*(my_j^m).
$$

Consequently, if we collect everything, we find that
$$
\liminf_{n\to\infty} \frac1n\log\PP(\sigma_n^n \in U) \geq -\frac1m\sum_{j=1}^m \Lambda^*(m\tilde y_j^m).
$$

Finally, using the convexity of $\Lambda^*$ and Jensen's inequality, we find that
$$
\frac1m\sum_{j=1}^m \Lambda^*(m\tilde y_j^m) \leq \sum_{i=1}^m \int_{\frac{i-1}{m}}^{\frac im} \Lambda^*(\dot\gamma(t)) \ud t = \int_0^1 \Lambda^*(\dot\gamma(t)) \ud t.
$$
From this, we then conclude that
$$
\liminf_{n\to\infty} \frac1n\log\PP(\sigma_n^n \in U) \geq -\int_0^1 \Lambda^*(\dot\gamma(t)) \ud t,
$$
which finishes the proof.

\subsection{Example: Products of transition matrices}

We conclude this section by discussing an example. In this example, we aim to study the limiting behaviour of products of transition matrices on a finite dimensional state space. For this we use the stochastic group and its Lie algebra, see e.g. \cite{GS18,Poo95}. For theory regarding matrix Lie groups, see e.g. \cite{Hal15}.\\

We define the set of transition matrices $\Tt(d,\RR)$ on $d$ states by
$$
\Tt(d,\RR) = \{P \in M(d,\RR)| P\mathbf{1} = \mathbf{1}, P_{ij} \geq 0 \mbox{ for } 1 \leq i,j \leq d\}.
$$
Here, $M(d,\RR)$ denotes the set of all $d\times d$-matrices, and $\mathbf{1}$ is the vector of all ones. Because we will be working with groups, we need inverses to be well-defined. We therefore consider the subset $\Ss_+(d,\RR)$ of invertible matrices in $\Tt(d,\RR)$, i.e.
$$
\Ss_+(d,\RR) = \{P \in \Tt(d,\RR)| \det(P) \neq 0\}.
$$
Note that $\Ss_+(d,\RR)$ is closed under matrix multiplication. Indeed, if $P$ and $Q$ have non-negative entries, then so does $PQ$. Furthermore, if $P\mathbf1 = \mathbf1$ and $Q\mathbf1 = \mathbf1$ then $PQ\mathbf1 = P\mathbf1 = \mathbf1$. Finally, if $P$ and $Q$ are invertible, then so is $PQ$. However, inverses of elements in $\Ss_+(d,\RR)$ need not have only non-negative entries. It turns out that the smallest group containing $\Ss_+(d,\RR)$ is given by
$$
\Ss(d,\RR) = \{P \in M(d,\RR)| \det(P) \neq 0, P\mathbf 1 = \mathbf 1\}.
$$ 
This group is called the \emph{stochastic group}. It is in fact a Lie group. Because we are dealing with matrix Lie groups, this follows from the observation that if $P_n \to P$ elementwise, and $P_n\mathbf1 = \mathbf 1$ for all $n$, then also $P\mathbf1 = \mathbf1$. \\

The Lie algebra associated to $\Ss(d,\RR)$ is given by
$$
\fs(d,\RR) = \{A \in M(d,\RR)|A\mathbf1 = 0\}.
$$
Indeed, if $A \in M(d,\RR)$ is such that $A\mathbf1 = 0$, then 
$$
\exp(tA)\mathbf1 = \mathbf1 + \left(\sum_{n=1}^\infty \frac{t^nA^{n-1}}{n!}\right)A\mathbf 1 = \mathbf 1.
$$
Consequently, $\exp(tA) \in \Ss(d,\RR)$ for all $t$, implying that 
$$
\{A \in M(d,\RR)|A\mathbf1 = 0\} \subset \fs(d,\RR).
$$
Conversely, if $\exp(tA)\mathbf1 = \mathbf1$ for all $t \in \RR$, then 
$$
A\mathbf1 = \frac{\dd}{\dd t}\Big|_{t=0} \exp(tA)\mathbf 1 = 0,
$$
so that
$$
\fs(d,\RR) \subset \{A \in M(d,\RR)|A\mathbf1 = 0\}.
$$

~\\

In order to consider random walks in the Lie group $\Ss(d,\RR)$ which only use invertible transition matrices, i.e., elements from $\Ss_+(d,\RR)$, we need to find a subset of $\fs(d,\RR)$ which is mapped by the exponential map into $\Ss_+(d,\RR)$. To this end, consider the set
$$
\fs_+(d,\RR) = \{A \in \fs(d,\RR)| A_{ij} \geq 0 \mbox{ whenever } i\neq j\}.
$$
We will prove that for all $A \in \fs_+(d,\RR)$ we have $\exp(A) \in \Ss_+(d,\RR)$. For this, it suffices to prove that $\exp(A)$ has nonnegative entries. To show this, we fix $k = \max_{i=1}^d |A_{ii}|$. Then the matrix $B = A + kI$ has nonnegative entries, from which it follows, using the Taylor series expression, that $\exp(B)$ has nonnegative entries. Because $A$ and $I$ commute, we have
$$
\exp(B) = \exp(A)\exp(kI) = e^k\exp(A),
$$
so that $\exp(A) = e^{-k}\exp(B)$. The latter now has nonnegative entries because $e^{-k} > 0$ and $\exp(B)$ has nonnegative entries.\\

Consequently, if we take a measure $\mu$ on $\fs(d,\RR)$ supported in $\fs_+(d,\RR)$, then the random walk $\Ss_n$ associated to an i.i.d. sequence $\{X_n\}_{n\geq1}$ will remain in $\Ss_+(\RR,d)$. This random walk may be thought of as the (random) $n$-step transition matrix of a Markov process with state space $\Omega = \{1,\ldots,d\}$.

From an increment $A \in \fs_+(d,\RR)$ of such a random walk, we can deduce some qualitative behaviour of the random walk. Indeed, for a state $i \in \{1,\ldots,d\}$ we have that the larger $|A_{ii}|$, the more mass remains at site $i$ after that iteration. The remainder of the mass at state $i$ is then distributed over the states $j \neq i$ according to the relative size of the $A_{ij}$.

\subsubsection{A specific example}

To get a better understanding of these random walks in $\Ss(d,\RR)$ and their limiting behaviour, we do the calculations for a specific example. For this, we take $d = 2$ and $\alpha,\beta > 0$. Consider the matrices
$$
A =
\left(\begin{array}{cc}
-\alpha & \alpha \\
0 & 0\\
\end{array}\right), \qquad
B = 
\left(\begin{array}{cc}
0 & 0\\
\beta & -\beta\\
\end{array}\right).
$$

Let $\{X_n\}_{n\geq1}$ be a sequence of i.i.d. random variables with $\PP(X_1 = A) = \PP(X_1 = B) = \frac12$. One may compute the exponentials of the matrices $A$ and $B$ to find
$$
\exp\left(\frac1nA\right) =
\left(\begin{array}{cc}
e^{-\frac1n\alpha} & 1 - e^{-\frac1n\alpha} \\
0 & 1\\
\end{array}\right), \qquad
\exp\left(\frac1nB\right) = 
\left(\begin{array}{cc}
1 & 0\\
1 - e^{-\frac1n\beta} & e^{-\frac1n\beta}\\
\end{array}\right).
$$

Intuitively, this process chooses one of the states uniformly at random and then distributes the mass at that state over the two states according to some parameter. Additionally, one sees that if $n$ tends to infinity, then the mass that is passed between states becomes exponentially small.

Now consider the rescaled random walk
$$
\sigma_n^n = \exp\left(\frac1nX_1\right)\cdots\exp\left(\frac1nX_n\right).
$$
By Theorem \ref{theorem:Cramer_Lie}, the sequence $\{\sigma_n^n\}_{n\geq1}$ satisfies in $\Ss(2,\RR)$ the large deviation principle. In order to quantify the rate function, we need to equip $\fs(2,\RR)$ with an inner product. For this, we will use the Frobenius inner product given by
$$
\inp{A}{B} = \Tr(A^TB) = \sum_{i,j=1}^2 A_{ij}B_{ij}.
$$

With this inner product, the log moment generating function $\Lambda: \fs(2,\RR) \to \RR$ of $X_1$ is given by
$$
\Lambda\left(\left(\begin{array}{cc}
-\lambda_1 & \lambda_1 \\
\lambda_2 & -\lambda_2\\
\end{array}\right)\right)
=
\log\left(\frac12e^{2\alpha\lambda_1} + \frac12e^{2\beta\lambda_2}\right).
$$

Let us compute $\Lambda^*:\fs(2,\RR) \to \RR$, i.e., we want to compute
\begin{align*}
\Lambda^*\left(\left(\begin{array}{cc}
-x_1 & x_1 \\
x_2 & -x_2\\
\end{array}\right)\right)
&=
\sup_{\lambda \in \fs(2,\RR)} \inp{\lambda}{x} - \Lambda(\lambda)
\\
&=
\sup_{\lambda_1,\lambda_2 \in \RR} 2\lambda_1x_1 + 2\lambda_2x_2 - \log\left(\frac12e^{2\alpha\lambda_1} + \frac12e^{2\beta\lambda_2}\right).
\end{align*}
Here we used that every $\lambda \in \fs(2,\RR)$ may be characterized by two elements $\lambda_1,\lambda_2 \in \RR$.

By taking $\lambda_2 = 0$ and letting $|\lambda_1|$ tend to infinity, we see that $\Lambda^*$ is infinite whenever $x_1 \notin [0,\alpha]$. In a similar way one can show that $\Lambda^*$ is infinite if $x_2 \notin [0,\beta]$. 

Next, we show that $\Lambda^*$ is also infinite if $\alpha x_2 + \beta x_1 \neq \alpha\beta$. To see this, take $\lambda_1,\lambda_2$ such that $\alpha\lambda_1 - \beta\lambda_2 = \alpha\beta$. Writing everything in terms of $\lambda_2$, we find that
\begin{align*}
\Lambda^*\left(\left(\begin{array}{cc}
-x_1 & x_1 \\
x_2 & -x_2\\
\end{array}\right)\right)
&\geq
2\lambda_1x_1 + 2\lambda_2x_2 - \log\left(\frac12e^{2\alpha\lambda_1} + \frac12e^{2\beta\lambda_2}\right)
\\
&=
2x_1\left(\beta + \frac\beta\alpha\lambda_2\right) + 2\lambda_2x_2 - \log\left(\frac12e^{2\beta\lambda_2}\left(e^{\alpha\beta} + 1\right)\right)
\\
&=
2\left(\frac\beta\alpha x_1 + x_2 - \beta\right)\lambda_2 + 2x_1 - \log\left(\frac12\left(e^{\alpha\beta} + 1\right)\right).
\end{align*}

Now by letting $|\lambda_2|$ tend to infinity, we see that, when maximized over $\lambda_2 \in \RR$, the above is only finite when 
$$
\frac\beta\alpha x_1 + x_2 - \beta = 0,
$$
which is equivalent to
$$
\beta x_1 + \alpha x_2 = \alpha\beta.
$$
~\\

Let us now compute the finite values of $\Lambda^*$. To this end, first consider the case $x_1 \in (0,\alpha)$ and $x_2 \in (0,\beta)$ with $\beta x_1 + \alpha x_2 = \alpha\beta$. Let us define 
$$
F(\lambda_1,\lambda_2) = 2\lambda_1x_1 + 2\lambda_2x_2 - \log\left(\frac12e^{2\alpha\lambda_1} + \frac12e^{2\beta\lambda_2}\right).
$$
Computing the gradient, and equating to 0, we find for the critical points of $F$ that
$$
x_1 = \frac{\alpha}{e^{2\alpha\lambda_1} + e^{2\beta\lambda_2}}e^{2\alpha\lambda_1}
$$
and
$$
x_2 = \frac{\beta}{e^{2\alpha\lambda_1} + e^{2\beta\lambda_2}}e^{2\beta\lambda_2}.
$$
Using that $\beta x_1 + \alpha x_2 = \alpha\beta$, we find that the above set of equations is solved by 
$$
\lambda_1^* = \frac{1}{2\alpha}\log(\beta x_1), \qquad \lambda_2^* = \frac{1}{2\beta}\log(\alpha x_2).
$$
Consequently, we find that
\begin{align*}
\Lambda^*\left(\left(\begin{array}{cc}
-x_1 & x_1 \\
x_2 & -x_2\\
\end{array}\right)\right)
&= 
F(\lambda_1^*,\lambda_2^*)
\\
&=
\frac1\alpha \log(\beta x_1)x_1 + \frac1\beta \log(\alpha x_2)x_2 - \log\left(\frac12\beta x_1 + \frac12\alpha x_2\right)
\\
&=
\frac1\alpha \log(\beta x_1)x_1 + \frac1\beta \log(\alpha x_2)x_2 - \log\left(\frac12\alpha\beta\right),
\end{align*}
where in the final step we used again that $\beta x_1 + \alpha x_2 = \alpha\beta$.

Now, in the case that $x_1 = 0$ and consequently, $x_2 = \beta$, we have
\begin{align*}
\Lambda^*\left(\left(\begin{array}{cc}
0 & 0 \\
\beta & -\beta\\
\end{array}\right)\right)
&=
\sup_{\lambda_2\in\RR} \sup_{\lambda_1\in\RR} \left\{2\lambda_2\beta - \log\left(\frac12e^{2\alpha\lambda_1} + \frac12e^{2\beta\lambda_2}\right)\right\}
\\
&=
\sup_{\lambda_2\in\RR} \left\{2\lambda_2\beta - \log\left(\frac12e^{2\beta\lambda_2}\right)\right\}
\\
&=
\log(2).
\end{align*}
Likewise, we also have
$$
\Lambda^*\left(\left(\begin{array}{cc}
-\alpha & \alpha \\
0 & 0\\
\end{array}\right)\right)
= \log(2).
$$

Now, the rate function for the large deviation principle for $\{\sigma_n^n\}_{n\geq1}$ is given by
$$
I(M) = \inf\left\{\int_0^1 \Lambda^*(\dot\gamma(t))\ud t \middle | \gamma \in \Aa\Cc([0,1];\Ss(2,\RR)), \gamma(0) = I, \gamma(1) = M\right\}.
$$

To get a more specific expression, we calculate the rate function further in the case where $\alpha = \beta$. Let $\gamma \in \Aa\Cc([0,1];\Ss(2,\RR))$ with $\gamma(0) = I$. Then we can write
$$
\gamma(t) = 
\left(\begin{array}{cc}
1 - \gamma_1(t) & \gamma_1(t) \\
\gamma_2(t) & 1 - \gamma_2(t)\\
\end{array}\right),
$$
so that
$$
\dot\gamma(t) = 
\left(\begin{array}{cc}
-\dot\gamma_1(t) & \dot\gamma_1(t) \\
\dot\gamma_2(t) & -\dot\gamma_2(t)\\
\end{array}\right) \in T_{\gamma(t)}\Ss(2,\RR).
$$
Now recall that we may identify $T_{\gamma(t)}\Ss(2,\RR)$ with $\fs(2,\RR)$ using the map $\dd L_{\gamma(t)}^{-1} = \dd L_{\gamma(t)^{-1}}$. Because $\Ss(2,\RR)$ is a matrix Lie group, we have
$$
\dd L_{\gamma(t)^{-1}}(X) = \gamma(t)^{-1}X.
$$
Consequently, as element of $\fs(2,\RR)$, the curve tangent to $\gamma$ is given by
\begin{align*}
& \dd L_{\gamma(t)^{-1}}(\dot\gamma) = \gamma(t)^{-1}\dot\gamma(t)
\\
&= 
\frac{1}{1 - \gamma_1(t) - \gamma_2(t)}\left(\begin{array}{cc}
-(1 - \gamma_2(t))\dot\gamma_1(t) - \gamma_1(t)\dot\gamma_2(t) & (1 - \gamma_2(t))\dot\gamma_1(t) + \gamma_1(t)\dot\gamma_2(t) \\
(1 - \gamma_1(t))\dot\gamma_2(t) + \gamma_2(t)\dot\gamma_1(t) & -(1 - \gamma_1(t))\dot\gamma_2(t) - \gamma_2(t)\dot\gamma_1(t)\\
\end{array}\right)
\end{align*}

Now, in order for 
$$
\int_0^1 \Lambda^*(\gamma(t)^{-1}\dot\gamma(t))\ud t
$$
to be finite, we need to have
$$
\beta\frac{(1 - \gamma_2(t))\dot\gamma_1(t) + \gamma_1(t)\dot\gamma_2(t)}{1 - \gamma_1(t) - \gamma_2(t)} + \alpha\frac{(1 - \gamma_1(t))\dot\gamma_2(t) + \gamma_2(t)\dot\gamma_1(t)}{1 - \gamma_1(t) - \gamma_2(t)} = \alpha\beta,
$$
because, as we have seen above, only then $\Lambda^*(\gamma(t)^{-1}\dot\gamma(t)) < \infty$.
Because $\alpha = \beta$, after some calculations, the above may be rewritten as
$$
\dot\gamma_1(t) + \dot\gamma_2(t) = \alpha(1 - (\gamma_1(t) + \gamma_2(t))).
$$
If we now write $\psi(t) = \gamma_1(t) + \gamma_2(t)$, the previous equality gives a differential equation for $\psi$, namely 
$$
\dot\psi(t) = \alpha(1 - \psi(t)),
$$
with $\psi(0) = \gamma_1(0) + \gamma_2(0) = 0$. Consequently, 
$$
\psi(t) = 1 - e^{-\alpha t}.
$$
In particular, this implies that 
$$
\gamma_1(1) + \gamma_2(1) = \psi(1) = 1 - e^{-\alpha}.
$$
From this, we deduce that $I(M)$ is only finite for matrices satisfying $M_{12} + M_{21} = 1 - e^{-\alpha}$. Now, if $M$ is such a matrix, the convexity of $\Lambda^*$ together with Jensen's inequality, implies that
$$
\inf\left\{\int_0^1 \Lambda^*(\dot\gamma(t))\ud t \middle | \gamma \in \Aa\Cc([0,1];\Ss(2,\RR)), \gamma(0) = I, \gamma(1) = M\right\}
$$
is attained when taking $\gamma_1^*(t) = c\psi(t)$ and $\gamma_2^*(t) = (1 - c)\psi(t)$. Because we need that $\gamma_1^*(1) = M_{12}$, we take
$$
\gamma_1^*(t) = \frac{M_{12}}{M_{12} + M_{21}}\psi(t) = \frac{M_{12}}{1 - e^{-\alpha}}(1 - e^{-\alpha t}),
$$
in which case
$$
\gamma_2^*(t) = \frac{M_{21}}{1 - e^{-\alpha}}(1 - e^{-\alpha t}).
$$

Using the expression for $\Lambda^*$ we derived above, one obtains after some computations that
\begin{align*}
I(M)
&=
\int_0^1 \Lambda^*\left(
\left(\begin{array}{cc}
-\dot\gamma^*_1(t) & \dot\gamma^*_1(t) \\
\dot\gamma^*_2(t) & -\dot\gamma^*_2(t)\\
\end{array}\right)\right) \ud t
\\
&= 
\alpha^2M_{12}\log\left(\frac{\alpha M_{12}}{1 - e^{-\alpha}}\right) - M_{12} - \frac{\alpha e^{-\alpha}M_{12}}{1 - e^{-\alpha}}
\\
&\qquad + 
\alpha^2M_{21}\log\left(\frac{\alpha M_{21}}{1 - e^{-\alpha}}\right) - M_{21} - \frac{\alpha e^{-\alpha}M_{21}}{1 - e^{-\alpha}} - \log\left(\frac12\alpha^2\right)
\\
&=
\alpha^2M_{12}\log\left(\frac{\alpha M_{12}}{1 - e^{-\alpha}}\right) + \alpha^2M_{21}\log\left(\frac{\alpha M_{21}}{1 - e^{-\alpha}}\right)
\\
&\qquad + 
(1 - \alpha) e^{-\alpha} - \log\left(\frac12\alpha^2\right) - 1
\end{align*}
if $M_{12} + M_{21}  =  1 - e^{-\alpha}$. Else, we have $I(M) = \infty$.

\section{Some estimation results from Lie group theory}\label{section:estimates}

In this section we use the integral version of the Baker-Campbell-Hausdorff formula to derive a key estimate we need for proving Theorem \ref{theorem:Cramer_Lie}. Essentially, we will show that for $X,Y \in \fg$ small enough, we can bound the difference between $\log(\exp(X)\exp(Y))$ and $X + Y$. In order to do this, we first introduce the integral version of the Baker-Campbell-Hausdorff formula.

\subsection{Baker-Campbell-Hausdorff formula}

Before we can state the Baker-Campbell-Hausdorff formula, we first need to introduce some linear operators on $\fg$. 

For every $X \in \fg$, we define the adjoint map $\ad_X:\fg \to \fg$ by 
$$
\ad_X(Y) := [X,Y].
$$

Because the map $(X,Y) \mapsto \ad_XY$ is smooth, it follows that $||\ad_X||$ depends continuously on $X$. In particular, this implies that 
$$
\sup_{X \in K} ||\ad_X|| < \infty
$$ 
for all $K \subset \fg$ compact. Additionally, it also gives us that
\[ \label{eq:small_adjoint}
\lim_{X \to 0} ||\ad_X|| = ||\ad_0|| = 0. 
\]

Because $\ad_X$ is a bounded operator, we can define the operator $e^{t\ad_X}$ by 
$$
e^{t\ad_X} = \sum_{m=0}^\infty \frac{t^m\ad_X^m}{m!}
$$
Similarly, for $f(z) = \frac{1 - e^{-z}}{z} = \sum_{m=0}^\infty \frac{(-1)^m}{(m+1)!}z^m$ we define the operator
\begin{equation}\label{eq:series_operator_diff}
\frac{I - e^{-\ad_X}}{\ad_X} = f(\ad_X) = \sum_{m=0}^\infty \frac{(-1)^m}{(m+1)!}\ad_X^m
\end{equation}

From this series repersentation, we find that
$$
\left|\left|I - \frac{I - e^{-\ad_X}}{\ad_X}\right|\right| \leq \sum_{k=1}^\infty \frac{||\ad_X||^k}{(k+1)!} \leq e^{||\ad_X||} - 1.
$$
Now, by \eqref{eq:small_adjoint} the upper bound goes to 0 if $X \to 0$. Consequently, if $|X|$ is small enough, then
$$
\frac{I - e^{-\ad_X}}{\ad_X}
$$
is invertible, with inverse given by
\begin{equation}\label{eq:series_inverse}
\frac{\ad_X}{I- e^{-\ad_X}} = g(e^{\ad_X})
\end{equation}
where $g(z) = \frac{z\log(z)}{z - 1} = 1 + \sum_{m=1}^\infty \frac{(-1)^{m+1}}{m(m+1)}(z-1)^m$ for $|z-1| < 1$.

With all relevant operators defined, we can state the integral form of the Baker-Campbell-Hausdorff formula, see e.g. \cite{Hal15,Vaj84}.

\begin{theorem}\label{theorem:BCH}
There exists an $r > 0$ such that for all $X,Y \in \fg$ with $|X|,|Y| \leq r$ we have that $\log(\exp(X)\exp(tY))$ is well-defined for all $t \in [0,1]$ and is given by
$$
\log(\exp(X)\exp(tY)) = X + \left(\int_0^t g(e^{\ad_X}e^{s\ad_Y})\ud s\right)Y,
$$
where $g(z) = \frac{z\log(z)}{z - 1} = 1 + \sum_{m=1}^\infty \frac{(-1)^{m+1}}{m(m+1)}(z-1)^m$ for $|z-1| < 1$
\end{theorem}

We will now use this formula to deduce approximations for the logarithm of a product of exponentials.

\subsection{Logarithm of a product of exponentials}

In this section, we aim to control the difference
$$
\left|\log(\exp(X)\exp(Y)) - X - Y\right|,
$$ 
for $X$ and $Y$ small enough. We will do this using the Baker-Campbell-Hausdorff formula. We have the following proposition.

\begin{proposition}\label{prop:log_product}
There exists an $r > 0$ and for every $X \in \fg$ with $|X| \leq r$ a constant $C_X > 0$ such that $\lim_{X \to 0} C_X = 0$ and
$$
\left|\log(\exp(X)\exp(Y)) - X - Y\right| \leq C_X|Y|
$$
for all $|Y| \leq r$. Moreover, the constant $C_X$ may be chosen to only depend on $|X|$.
\end{proposition}
\begin{proof}
By Theorem \ref{theorem:BCH} we have
$$
\log(\exp(X)\exp(Y)) = X + Y + \left(\int_0^1 \sum_{m=1}^\infty \frac{(-1)^m}{m(m+1)}(e^{\ad_X}e^{s\ad_Y} - I)^m \ud s\right)Y.
$$
Consequently, we have
\begin{align*}
&|\log(\exp(X)\exp(Y)) - X - Y|
\\
&=
\left|\left(\int_0^1 \sum_{m=1}^\infty \frac{(-1)^m}{m(m+1)}(e^{\ad_X}e^{s\ad_Y} - I)^m \ud s\right)Y\right|
\\
&\leq
\int_0^1 \sum_{m=1}^\infty \frac{1}{m(m+1)}||e^{\ad_X}e^{s\ad_Y} - I||^{m-1}|(e^{\ad_X}e^{s\ad_Y} - I)Y| \ud s.
\end{align*}

Because $\ad_YY = 0$, we find that
$$
e^{s\ad_Y}Y = Y + \sum_{m=1}^\infty \frac{s^m\ad_Y^{m-1}}{m!}\ad_YY = Y,
$$
so that
$$
|(e^{\ad_X}e^{s\ad_Y} - I)Y| = |(e^{\ad_X} - I)Y| \leq ||e^{\ad_X} - I|||Y| \leq (e^{||\ad_X||} - 1)|Y|.
$$
Here, the latter follows from
$$
||e^{\ad_X} - I|| \leq \sum_{m=1}^\infty \frac{||\ad_X||^m}{m!} = e^{||\ad_X||} - 1.
$$

Now define $Z(t) = \log(\exp(X)\exp(tY))$. Then (see e.g. \cite[Chapter 5]{Hal15} or \cite[Chapter 2]{Vaj84})
$$
e^{\ad_X}e^{s\ad_Y} = e^{\ad_{Z(s)}},
$$
see e.g. \cite[Chapter 5]{Hal15} or \cite[Chapter 2]{Vaj84}. From this we deduce
$$
||e^{\ad_X}e^{s\ad_Y} - I|| \leq e^{||\ad_{Z(s)}||} - 1.
$$

By \eqref{eq:small_adjoint}, we find $r' > 0$ such that $||\ad_{Z(s)}|| \leq \frac{\log(2)}{2}$ whenever $|Z(s)| \leq r'$. By Proposition \ref{prop:injectivity}, there is $r'' > 0$ such that this in turn follows from $d(e,\exp(Z(s))) \leq r''$.

Now we have
\begin{align*}
d(e,\exp(Z(s)))
&=
d(e,\exp(X)\exp(tY))
\\
&\leq 
d(e,\exp(X)) + d(\exp(X),\exp(X)\exp(tY)) 
\\
&= d(e,\exp(X)) + d(e,\exp(tY)) \label{eq:distance_estimate}
\\
&\leq 
|X| + t|Y|,
\end{align*}
where we used the triangle inequality and left-invariance of the metric. The last step follows by noticing that if $\gamma(t) = \exp(tX)$, then 
$$
d(e,\exp(X)) \leq \int_0^1 |\dot\gamma(t)|\ud t = |X|.
$$

Consequently, if $|X|,|Y| \leq \frac12r''$, then $d(e,\exp(Z(s))) \leq r''$ so that $||\ad_{Z(s)}|| \leq \frac{\log(2)}{2}$ for all $s \in (0,1)$. But then $||e^{\ad_X}e^{s\ad_Y} - I|| \leq \sqrt2 - 1 < 1$, and hence
\begin{align*}
&\int_0^1 \sum_{m=1}^\infty \frac{1}{m(m+1)}||e^{\ad_X}e^{s\ad_Y} - I||^{m-1}|(e^{\ad_X}e^{s\ad_Y} - I)Y| \ud s
\\
&\leq
(e^{||\ad_X||} - 1)|Y|\int_0^1 \sum_{m=1}^\infty \frac{(\sqrt2 - 1)^{m-1}}{m(m+1)} \ud s
\\
&\leq
(e^{||\ad_X||} - 1)|Y|\sum_{m=1}^\infty \frac{(\sqrt2 - 1)^{m-1}}{m(m+1)}
\end{align*}

Consequently, we may take 
$$
C_X = (e^{||\ad_X||} - 1)\sum_{m=1}^\infty \frac{(\sqrt2 - 1)^{m-1}}{m(m+1)} < \infty.
$$
Because $\lim_{|X| \to 0} ||\ad_X|| = 0$, it follows that $\lim_{X \to 0} C_X = 0$, and that $C_X$ may be chosen to depend only on $|X|$.

\end{proof}

We conclude this section with the following result, which shows a Lipschitz-like estimate for the logarithm of a product of two exponentials.

\begin{proposition}\label{prop:log_Lipschitz}
There exist constants $r > 0$ and $C > 0$ such that for all $X,Y \in \fg$ with $|X|,|Y| \leq r$ we have
$$
|\log(\exp(X)\exp(-Y))| \leq C|X - Y|
$$
\end{proposition}
\begin{proof}
Following the same reasoning as in the proof of Proposition \ref{prop:log_product} we obtain
$$
\log(\exp(X)\exp(-Y)) 
=
X - Y - \left(\int_0^1 \sum_{m=1}^\infty \frac{(I - e^{\ad_X}e^{-t\ad_Y})^m}{m(m+1)} \ud t\right)Y.
$$

As before, we have $e^{-t\ad_Y}Y = Y$ and similarly $e^{\ad_X}X = X$. Consequently, we can write
$$
(I - e^{\ad_X}e^{-t\ad_Y})Y = (I - e^{\ad_X})Y = (I - e^{\ad_X})(Y - X),
$$
from which it follows that
$$
|(I - e^{\ad_X}e^{-t\ad_Y})^mY| \leq ||I - e^{\ad_X}e^{-t\ad_Y}||^{m-1}||I - e^{\ad_X}|||Y - X|.
$$

By similar reasoning as in the proof of Proposition \ref{prop:log_product}, we can find $r > 0$ and a constant $C,\tilde C > 0$ such that $|X|,|Y| \leq r$ implies that
$$
\left|\left(\int_0^1 \sum_{m=1}^\infty \frac{(I - e^{\ad_X}e^{-t\ad_Y})^m}{m(m+1)} \ud t\right)Y\right| \leq \tilde C||I - e^{\ad_X}|||Y - X| \leq C|X - Y|.
$$
By the triangle inequality we then find that
$$
|\log(\exp(X)\exp(-Y))| \leq (C+1)|X - Y|
$$
as desired.

\end{proof}


\section{Proof of Theorem \ref{theorem:Cramer_Lie}}\label{section:proof}

In this section we provide a proof of Theorem \ref{theorem:Cramer_Lie}. As explained in Section \ref{section:sketch}, we prove the upper bound and lower bound for the large deviation principle of $\{\sigma_n^n\}_{n\geq0}$ seperately. More precisely, Theorem \ref{theorem:Cramer_Lie} follows immediately from Propositions \ref{prop:upper_bound_LDP} and \ref{prop:lower_bound_LDP}. Before we get to either of these, we first need two general results, which we use in both the proof of the upper and lower bound.

Before we get to the first results, let us define for every $n \in \NN$ and every $1 \leq k \leq n$ the random variable
$$
\sigma_k^n = \exp\left(\frac1nX_1\right)\cdots\exp\left(\frac1nX_k\right) \in G,
$$
i.e., the point of the rescaled random walk after $k$ increments. Finally, we set $\sigma_0^n = e$. We have the following estimate.

\begin{proposition}\label{prop:replacement_sum}
Let the assumptions Theorem \ref{theorem:Cramer_Lie} be satisfied. Then for every $m$ large enough, there exists a constant $C_m > 0$ with $\lim_{m\to\infty} C_m = 0$ such that for all $1 \leq k \leq \lfloor m^{-1}n\rfloor$, $\log(\sigma^n_k)$ is well-defined and
$$
\left|\log\left(\sigma^n_k\right) - \frac1n\sum_{i=1}^kX_i\right| \leq C_m\frac1m.
$$
\end{proposition}
\begin{proof}
First note that by the triangle inequality we have for any $n$ and $1 \leq k \leq n$ that
$$
d(\sigma_k^n,e) \leq \sum_{i=1}^k d(\sigma_i^n,\sigma_{i-1}^n).
$$
Considering the curve $\gamma_i(t) = \sigma_{i-1}^n\exp(tX_i)$ in $G$, we obtain
$$
d(\sigma_i^n,\sigma_{i-1}^n) \leq \int_0^{\frac1n} |\dot\gamma_i(t)| \ud t = \frac1n|X_i|.
$$
Hence, if we write $B$ for the uniform bound on the increments, we find
\begin{equation}\label{eq:triangle_distance}
d(\sigma_k^n,e) \leq \frac{k}{n}B.
\end{equation}
But then we have for $1 \leq k \leq \lfloor m^{-1}n\rfloor$ that
\begin{equation}\label{eq:estimate_distance}
d(\sigma_k^n,e) \leq \frac{\lfloor m^{-1}n\rfloor}{n}B \leq \frac1mB.
\end{equation}
Thus if we choose $m$ large enough, we can assure that $\sigma_k^n$ is sufficiently close to $e$ for $k = 1,\ldots,\lfloor m^{-1}n\rfloor$, so that $\log(\sigma_k^n)$ is well-defined for $1 \leq k \leq \lfloor m^{-1}n\rfloor$.\\ 

Turning to the proof of the estimate, first note that we may write 
$$
\log\left(\sigma^n_k\right) = \sum_{i=1}^k \log\left(\sigma^n_i\right) - \log\left(\sigma^n_{i-1}\right)
$$
so that 
$$
\left|\log\left(\sigma^n_k\right) - \frac1n\sum_{i=1}^kX_i\right|
\leq
\sum_{i=1}^k \left|\log\left(\sigma^n_i\right) - \log\left(\sigma^n_{i-1}\right) - \frac1nX_i\right|.
$$

Now note that by Proposition \ref{prop:injectivity}, for every $r > 0$ there exists an $\epsilon > 0$ such that $d(e,g) \leq \epsilon$ implies that $|\log(g)| \leq r$. Consequently, it follows from \eqref{eq:estimate_distance} that for $1 \leq k \leq \lfloor m^{-1}n\rfloor$, $|\log(\sigma_k^n)|$ can be made arbitrarily small by taking $m$ large enough. Furthermore, because $|X_i| \leq B$, we find that $\frac1nX_i$ becomes small for large $n$. Consequently, for $m$ and $n$ large enough we can apply Proposition \ref{prop:log_product} to obtain constants $C_m$ with $\lim_{m\to\infty} C_m = 0$ such that
$$
\left|\log\left(\sigma^n_i\right) - \log\left(\sigma^n_{i-1}\right) - \frac1nX_i\right| \leq C_m\frac1n|X_i| \leq C_m\frac1nB
$$

Combining everything, we find that
$$
\left|\log\left(\sigma^n_k\right) - \frac1n\sum_{i=1}^kX_i\right| \leq \sum_{i=1}^k C_m\frac1n = C_m\frac kn \leq C_m\frac1m.
$$
Here we used that $k \leq \lfloor m^{-1}n\rfloor$ and absorbed the constant $B$ into $C_m$. 

\end{proof}

Note that we do in general not have that $\log(\sigma_k^n)$ exists in $\fg$ for all $n$ and all $1 \leq k \leq n$. Consequently, in order to be able to use some identification of the random walk with a process in the Lie algebra, we need to make sure we can actually use the logarithm map.

To this end, notice that in the previous proof, we have see in \eqref{eq:estimate_distance} that for $1 \leq k \leq \lfloor m^{-1}n\rfloor$ we have $d(\sigma_k^n,e) \leq \frac1mB$, where $B$ is the uniform bound on the increments. 

With this estimate in mind, the idea is now to split the random walk into $m$ pieces, each consisting of (approximately) $\lfloor m^{-1}n\rfloor$ increments. More precisely, we define the indices $n_l = l\lfloor m^{-1}n\rfloor$ for $l = 0,\ldots, m-1$ and set $n_m = n$. Because the metric is left-invariant, we have for every $l = 1,\ldots,m$ and every $k = 1,\ldots,n_l - n_{l-1}$ that
\begin{equation}\label{eq:distance_pieces}
d(e,(\sigma^n_{n_{l-1}})^{-1}\sigma^n_{n_{l-1}+k}) = d(\sigma^n_{n_{l-1}},\sigma^n_{n_{l-1}+k}) \leq \frac1mB,
\end{equation}
where $B$ is the uniform bound of the increments, the estimate following in the same way as we obtained \eqref{eq:estimate_distance}. Consequently, for $m$ large enough we can define
$$
Y^{n,m,l}_k = \log\left((\sigma^n_{n_{l-1}})^{-1}\sigma^n_{n_{l-1}+k}\right) \in \fg.
$$ 
for every $l = 1,\ldots,m$ and $k = 1,\ldots,n_l - n_{l-1}$.

Note that
$$
(\sigma^n_{n_{l-1}})^{-1}\sigma^n_{n_{l-1}+k} = \exp\left(\frac1nX_{n_{l-1}+1}\right)\cdots\exp\left(\frac1nX_{n_{l-1}+k}\right), 
$$
so that
\begin{equation}\label{eq:partition}
Y^{n,m,l}_k = \log\left(\exp\left(\frac1nX_{n_{l-1}+1}\right)\cdots\exp\left(\frac1nX_{n_{l-1}+k}\right)\right).
\end{equation}

Now for every $m$, this allows us to define a random vector 
\begin{equation}\label{eq:random_walk_vector}
\left(Y_{\lfloor m^{-1}n\rfloor}^{n,m,1},\ldots,Y_{\lfloor m^{-1}n\rfloor}^{n,m,m}\right) \in \fg^m.
\end{equation}
By \eqref{eq:partition}, we have that $Y_{\lfloor m^{-1}n\rfloor}^{m,1},\ldots,Y_{\lfloor m^{-1}n\rfloor}^{m,m}$ are independent and identically distributed random variables in $\fg$, because the $X_i$ are independent and identically distributed by assumption.

\subsection{Proof of the upper bound for the large deviation principle of $\{\sigma_n^n\}_{n\geq0}$}

In this section we prove the upper bound of the large deviation principle of $\{\sigma_n^n\}_{n\geq0}$. As explained in Section \ref{section:sketch_upper}, we do this by transferring the problem to the Lie algebra and obtain suitable estimates there using a similar approach as in the Euclidean case. We start with the following result.

\begin{proposition}\label{prop:upper_bound_algebra}
Let the assumptions of Theorem \ref{theorem:Cramer_Lie} be satisfied. Let $m \in \NN$ be large enough so that the random vector 
$$
\left(Y_{\lfloor m^{-1}n\rfloor}^{n,m,1},\ldots,Y_{\lfloor m^{-1}n\rfloor}^{n,m,m}\right) \in \fg^m
$$
defined in \eqref{eq:random_walk_vector} is well-defined. Then for every $F \subset \fg^m$ closed we have
\begin{align*}
&\limsup_{n\to\infty} \frac1n\log\PP\left(\left(Y_{\lfloor m^{-1}n\rfloor}^{n,m,1},\ldots,Y_{\lfloor m^{-1}n\rfloor}^{n,m,m}\right) \in F\right)
\\
&\leq 
-\inf_{(x_1,\ldots,x_m) \in F} \frac1m\sum_{i=1}^m \sup_{\lambda \in \fg} \left\{\inp{\lambda}{mx_i} - \Lambda(\lambda) - C_m|\lambda|\right\}.
\end{align*}
Here, $\Lambda(\lambda) = \log\EE(e^{{\lambda}{X_1}})$ and $C_m$ is a constant such that $\lim_{m\to\infty} C_m = 0$. 
\end{proposition}

\begin{proof}

Following the proof of Cram\'er's theorem for the vector space $\fg^m$ (see e.g. \cite{DZ98,Hol00}), we have for any $\Gamma \subset \fg^m$ compact that
\begin{align*}
&\limsup_{n\to\infty} \frac1n\log\PP\left(\left(Y_{\lfloor m^{-1}n\rfloor}^{n,m,1},\ldots,Y_{\lfloor m^{-1}n\rfloor}^{n,m,m}\right) \in \Gamma\right)
\\
&\leq 
-\inf_{(x_1,\ldots,x_m) \in \Gamma} \sup_{(\lambda_1,\ldots,\lambda_m) \in \fg^m} \left\{\sum_{i=1}^m \inp{\lambda_i}{x_i} - \limsup_{n\to\infty} \frac1n\log\EE\left(e^{n\sum_{i=1}^m\inp{\lambda_i}{Y_{\lfloor m^{-1}n\rfloor}^{n,m,i}}}\right)\right\}
\end{align*}

However, as mentioned above, the fact that the $X_i$ are independent and identically distributed, together with \eqref{eq:random_walk_vector}, shows that $Y_{\lfloor m^{-1}n\rfloor}^{n,m,1},\ldots,Y_{\lfloor m^{-1}n\rfloor}^{n,m,m}$ are independent and identically distributed. Hence
$$
\EE\left(e^{n\sum_{i=1}^m\inp{\lambda_i}{Y_{\lfloor m^{-1}n\rfloor}^{n,m,i}}}\right) = \prod_{i=1}^m \EE\left(e^{n\inp{\lambda_i}{Y_{\lfloor m^{-1}n\rfloor}^{n,m,1}}}\right).
$$
By Proposition \ref{prop:replacement_sum}, there exist constants $C_m > 0$ with $\lim_{m\to\infty} C_m = 0$ such that
$$
\left|Y_{\lfloor m^{-1}n\rfloor}^{n,m,1} - \frac1n\sum_{i=1}^{\lfloor m^{-1}n\rfloor} X_i\right| \leq C_m\frac1m.
$$
Consequently, using the Cauchy-Schwarz inequality, we have
\begin{align*}
\EE\left(e^{n\inp{\lambda_i}{Y_{\lfloor m^{-1}n\rfloor}^{n,m,1}}}\right) 
&\leq 
\EE\left(e^{\sum_{j=1}^{\lfloor m^{-1}n\rfloor}\inp{\lambda_i}{X_j}}\right)e^{n|\lambda_i|C_mm^{-1}}
\\
&=
e^{n|\lambda_i|C_mm^{-1}}\EE\left(e^{\inp{\lambda_i}{X_1}}\right)^{\lfloor m^{-1}n\rfloor}.
\end{align*}
Hence
\begin{align*}
\limsup_{n\to\infty} \frac1n\log\EE\left(e^{n\sum_{i=1}^m\inp{\lambda_i}{Y_{\lfloor m^{-1}n\rfloor}^{n,m,i}}}\right)
&=
\sum_{i=1}^m \limsup_{n\to\infty} \frac1n\log\EE\left(e^{n\inp{\lambda_i}{Y_{\lfloor m^{-1}n\rfloor}^{n,m,1}}}\right)
\\
&\leq
\sum_{i=1}^m \left\{ |\lambda_i|C_m\frac1m + \frac1m\log\EE\left(e^{\inp{\lambda_i}{X_1}}\right) \right\}
\\
&= 
\frac1m\sum_{i=1}^m \left\{C_m|\lambda_i| + \log \EE\left(e^{\inp{\lambda_i}{X_1}}\right) \right\}.
\end{align*}

Collecting everything, we find that
\begin{align*}
&\limsup_{n\to\infty} \frac1n\log\PP\left(\left(Y_{\lfloor m^{-1}n\rfloor}^{n,m,1},\ldots,Y_{\lfloor m^{-1}n\rfloor}^{n,m,m}\right) \in \Gamma\right)
\\
&\leq 
- \inf_{(x_1,\ldots,x_m) \in \Gamma} \sup_{(\lambda_1,\ldots,\lambda_m) \in \fg^m} \frac1m\sum_{i=1}^m\left\{ \inp{\lambda_i}{mx_i} - \log\EE\left(e^{\inp{\lambda_i}{X_1}}\right) - C_m|\lambda_i|\right\}
\\
&=
- \inf_{(x_1,\ldots,x_m) \in \Gamma} \frac1m\sum_{i=1}^m\sup_{\lambda \in \fg} \left\{ \inp{\lambda}{mx_i} - \Lambda(\lambda) - C_m|\lambda|\right\}.
\end{align*}

To extend this upper bound to all closed sets, note that the boundedness of the increments of the random walk implies that $Y_{\lfloor m^{-1}n\rfloor}^{n,m,1}$ is bounded, and hence remains in a compact subset of $\fg$. Because $Y_{\lfloor m^{-1}n\rfloor}^{n,m,1},\ldots,Y_{\lfloor m^{-1}n\rfloor}^{n,m,m}$ are independent and identically distributed, we can conclude from this that $(Y_{\lfloor m^{-1}n\rfloor}^{n,m,1},\ldots,Y_{\lfloor m^{-1}n\rfloor}^{n,m,m})$ is exponentially tight in $\fg^m$. From this it follows that the upper bound actually holds for all closed sets, which completes the proof.

\end{proof}

With the preparations done, we can now turn to the proof of the upper bound of the large deviation principle for $\{\sigma_n^n\}_{n\geq1}$. The main work goes into proving that we actually obtain the desired form of the upper bound.

\begin{proposition}\label{prop:upper_bound_LDP}
Let the assumptions of Theorem \ref{theorem:Cramer_Lie} be satisfied. Then for any $F \subset G$ closed we have
$$
\limsup_{n\to\infty} \frac1n\log\PP(\sigma_n^n \in F) \leq -\inf_{g\in F} I_G(g),
$$
where $I_G$ is the good rate function given by \eqref{eq:ratefunction}.
\end{proposition}

\begin{proof}
Let $F \subset G$ be closed. Choose $m \in \NN$ large enough so that the random vector 
$$
\left(Y_{\lfloor m^{-1}n\rfloor}^{n,m,1},\ldots,Y_{\lfloor m^{-1}n\rfloor}^{n,m,m}\right) \in \fg^m
$$
defined in \eqref{eq:random_walk_vector} is well-defined. Let $\Psi_m:\fg^m \to G$ be the map given by
$$
\Psi_m(x_1,\ldots,x_m) = \exp(x_1)\cdots\exp(x_m).
$$
Because $\Psi_m$ is a composition of continuous functions, it is itself continuous. Furthermore, observe that by construction
$$
\Psi_m\left(Y_{\lfloor m^{-1}n\rfloor}^{n,m,1},\ldots,Y_{\lfloor m^{-1}n\rfloor}^{n,m,m}\right) = \sigma_n^n.
$$
Consequently, we have
$$
\PP\left(\sigma_n^n \in F\right) \leq \PP\left(\left(Y_{\lfloor m^{-1}n\rfloor}^{n,m,1},\ldots,Y_{\lfloor m^{-1}n\rfloor}^{n,m,m}\right) \in \Psi_m^{-1}F\right),
$$
where $\Psi_m^{-1}F$ is closed, because $F$ is closed and $\Psi_m$ is continuous. By Proposition \ref{prop:upper_bound_algebra} we then find that
\begin{align*}
&\limsup_{n\to\infty} \frac1n\log\PP\left(\sigma_n^n \in F\right)
\\
&\leq
-\inf_{(x_1,\ldots,x_m) \in \Psi_m^{-1}F} \frac1m\sum_{i=1}^m \sup_{\lambda \in \fg} \left\{\inp{\lambda}{mx_i} - \Lambda(\lambda) - C_m|\lambda|\right\},
\end{align*}
where $\lim_{m\to\infty} C_m = 0$. 

The final step is now to let $m$ tend to infinity, and show that we obtain the desired upper bound. For this, we need to show that
$$
-\lim_{m\to\infty} \inf_{(x_1,\ldots,x_m) \in \Psi_m^{-1}F} \frac1m\sum_{i=1}^m \sup_{\lambda \in \fg} \left\{\inp{\lambda}{mx_i} - \Lambda(\lambda) - C_m|\lambda|\right\} \leq -\inf_{g \in F} I_G(g).
$$

To this end, let $\epsilon > 0$ be arbitrary. Because $\lim_{m\to\infty} C_m = 0$, we can find $m_0 \in \NN$ such that $m \geq m_0$ implies that $C_m < \epsilon$. In that case, we have
\begin{align*}
&-\inf_{(x_1,\ldots,x_m) \in \Psi_m^{-1}F} \frac1m\sum_{i=1}^m \sup_{\lambda \in \fg} \left\{\inp{\lambda}{mx_i} - \Lambda(\lambda) - C_m|\lambda|\right\} 
\\
&\leq
-\inf_{(x_1,\ldots,x_m) \in \Psi_m^{-1}F} \frac1m\sum_{i=1}^m \sup_{\lambda \in \fg} \left\{\inp{\lambda}{mx_i} - \Lambda(\lambda) - \epsilon|\lambda|\right\}
\\
&=
-\inf_{(x_1,\ldots,x_m) \in \Psi_m^{-1}F} \frac1m\sum_{i=1}^m \Lambda_\epsilon^*(mx_i),
\end{align*}
where $\Lambda_\epsilon(\lambda) = \Lambda(\lambda) + \epsilon|\lambda|$ and $\Lambda_\epsilon^*$ denotes its Legendre transform. 

Now note that 
$$
\frac1m\sum_{i=1}^m \Lambda_\epsilon^*(mx_i) = \int_0^1 \Lambda_\epsilon^*(\dot\gamma(t))\ud t,
$$
where $\gamma:[0,1] \to G$ is given by $\gamma(0) = e$ and
$$
\gamma(t) = \gamma\left(\frac{i-1}m\right)\exp\left(\left(t - \frac{i-1}{m}\right)mx_i\right), \qquad t \in \left[\frac{i-1}{m},\frac im\right],
$$
for $i = 1,\ldots,m$. Furthermore, note that $\gamma(1) = \Psi_m(x_1,\ldots,x_m)$. 

Using this, we find that
\begin{align*}
&-\inf_{(x_1,\ldots,x_m) \in \Psi_m^{-1}F} \frac1m\sum_{i=1}^m \Lambda_\epsilon^*(mx_i)
\\
&\leq
-\inf\left\{\int_0^1 \Lambda_\epsilon^*(\dot\gamma(t))\ud t| \gamma:[0,1] \to G, \gamma(0) = e, \gamma(1) = g, \gamma \in AC\right\}.
\end{align*}

It remains to consider the limit $\epsilon \to 0$. To this end, first suppose that $I_G(g) < \infty$. By the goodness of the ratefunction $\Ii_\epsilon(\gamma) = \int_0^1 \Lambda_\epsilon^*(\dot\gamma(t))\ud t$, the sets
$$
C_\epsilon := \left\{ \gamma \middle | \int_0^1 \Lambda_\epsilon^*(\dot\gamma(t))\ud t \leq 2I_G(g)\right\}
$$
are compact. Furthermore, we have $C_{\epsilon'} \subset C_\epsilon$ whenever $\epsilon' \leq \epsilon$. Because lower-semicontinuous functions attain their minimum on compact sets, we have a sequence $\gamma_\epsilon$ such that 
\begin{align*}
&\int_0^1 \Lambda_\epsilon^*(\dot\gamma_\epsilon(t))\ud t
\\
&= \inf\left\{\int_0^1 \Lambda_\epsilon^*(\dot\gamma(t))\ud t| \gamma:[0,1] \to G, \gamma(0) = e, \gamma(1) = g, \gamma \in AC\right\}
\\
&=: I_\epsilon.
\end{align*}
Because the sequence $C_\epsilon$ is decreasing, for $\epsilon$ small enough, the sequence $\gamma_\epsilon$ is contained in a compact set, and hence, upon passing to subsequences, we may assume that $\gamma_\epsilon$ converges with limit $\gamma$. But then we find for every $\delta > 0$ that
\begin{align*}
\liminf_{\epsilon \to 0} I_\epsilon
&=
\liminf_{\epsilon \to 0} \int_0^1 \Lambda_\epsilon^*(\dot\gamma_\epsilon(t))\ud t 
\\
&\geq 
\liminf_{\epsilon \to 0} \int_0^1 \Lambda_\delta^*(\dot\gamma_\epsilon(t))\ud t
\\
&\geq
\int_0^1 \Lambda_\delta^*(\dot\gamma(t))\ud t.
\end{align*}
As this holds for all $\delta > 0$, by taking the limit $\delta \to 0$ we find that
$$
\liminf_{\epsilon \to 0} I_\epsilon \geq \int_0^1 \Lambda^*(\dot\gamma(t))\ud t \geq I_G(g).
$$
Because also $I_\epsilon \leq I_G(g)$ for every $\epsilon > 0$, we find that $\lim_{\epsilon \to 0} I_\epsilon = I_G(g)$ as desired. 

Now consider the case that $I_G(g) = \infty$. Suppose that $I_\epsilon$ does not converge to $\infty$. Then $\liminf_{\epsilon \to 0} I_\epsilon < \infty$. Upon passing to subsequences, suppose that $\lim_{\epsilon \to 0} I_\epsilon = I$. Following a similar reasoning as above, we find a sequence $\gamma_\epsilon$ converging to $\gamma$ which we can use to show that
$$
I_G(g) \leq \liminf_{\epsilon \to 0} I_\epsilon < \infty,
$$
which is a contradiction. Consequently, we find that $\lim_{\epsilon \to 0} I_\epsilon = \infty$.

Collecting everything, we have that
$$
\lim_{\epsilon \to 0} \left[\inf\left\{\int_0^1 \Lambda_\epsilon^*(\dot\gamma(t))\ud t| \gamma:[0,1] \to G, \gamma(0) = e, \gamma(1) = g, \gamma \in AC\right\}\right] = I_G(g).
$$
so that
$$
\limsup_{n\to\infty} \frac1n\log\PP(\sigma_n^n \in F) \leq -\inf_{g\in F} I_G(g),
$$
as desired.
\end{proof}

\subsection{Proof of the lower bound for the large deviation principle of $\{\sigma_n^n\}_{n\geq0}$}

In this section we prove the lower bound for the large deviation principle of $\{\sigma_n^n\}_{n\geq0}$. Before we can do this, we first need to study more carefully the continuity properties of the maps $\Psi_m:\fg^m \to G$ given, as in the proof of Proposition \ref{prop:upper_bound_LDP}, by
$$
\Psi_m(x_1,\ldots,x_m) = \exp(x_1)\cdots\exp(x_m).
$$

We have the following lemma.


\begin{lemma}\label{lemma:bounded_conjugation}
Let $K \subset G$ be compact. Denote by $\Ad_g:\fg \to \fg$ conjugation by $g$, i.e., $\Ad_gX = gXg^{-1}$. Then 
$$
\sup_{g\in K} ||\Ad_g|| < \infty.
$$ 
\end{lemma}
\begin{proof}
For every $X \in \fg$, the map $g \mapsto \Ad_gX$ is continuous, and hence bounded on compact sets. The claim then follows from the uniform boundedness theorem.
\end{proof}

We can now prove the following continuity property of the maps $\Psi_m$.

\begin{proposition}\label{prop:continuity_repeated_exponential}
For every $r > 0$, there exists a constant $C > 0$ such that for all $\epsilon > 0$ and $m \in \NN$ large enough we have that if
$$
(x_1,\ldots,x_m) \in B(y_1,C^{-1}\epsilon) \times \cdots B(y_1,C^{-1}\epsilon),
$$
then
$$
\Psi_m(x_1,\ldots,x_m) \in B(\Psi_m(y_1,\ldots,y_m),\epsilon) 
$$
whenever $|x_i|,|y_i| \leq \frac rm$.
\end{proposition}
\begin{proof}
By the triangle inequality, we have
\begin{align*}
&d(\Psi_m(x_1,\ldots,x_m),\Psi_m(y_1,\ldots,y_m))
\\
&=
d(\Psi_{m-1}(x_2,\ldots,x_m),\exp(-x_1)\exp(y_1)\Psi_{m-1}(y_2,\ldots,y_m))
\\
&\leq
d(\Psi_{m-1}(x_2,\ldots,x_m),\Psi_{m-1}(y_2,\ldots,y_m))
\\
&\qquad +
d(\Psi_{m-1}(y_2,\ldots,y_m),\exp(-x_1)\exp(y_1)\Psi_{m-1}(y_2,\ldots,y_m))
\\
&=
d(\Psi_{m-1}(x_2,\ldots,x_m),\Psi_{m-1}(y_2,\ldots,y_m))
\\
&\qquad +
d(e,\Psi_{m-1}(y_2,\ldots,y_m)^{-1}\exp(-x_1)\exp(y_1)\Psi_{m-1}(y_2,\ldots,y_m)).
\end{align*}
Here we used in the second and fourth line that the metric is left-invariant, while we used the triangle inequality in the third line. 

Now notice that if $m$ is large enough, then for $x_1,y_1$ with $|x_1|,|y_1| \leq \frac rm$, we have that 
$$
\log(\exp(-x_1)\exp(y_1))
$$ 
is well-defined. Furthermore, by Proposition \ref{prop:log_Lipschitz} there exists a constant $C$ such that
$$
|\log(\exp(-x_1)\exp(y_1))| \leq C|x_1 - y_1|.
$$

Now notice that 
\begin{align*}
&\exp\left(\Psi_{m-1}(y_2,\ldots,y_m)^{-1}\log(\exp(-x_1)\exp(y_1))\Psi_{m-1}(y_2,\ldots,y_m)\right)
\\
&=
\Psi_{m-1}(y_2,\ldots,y_m)^{-1}\exp(\log(\exp(-x_1)\exp(y_1)))\Psi_{m-1}(y_2,\ldots,y_m)
\\
&=
\Psi_{m-1}(y_2,\ldots,y_m)^{-1}\exp(-x_1)\exp(y_1)\Psi_{m-1}(y_2,\ldots,y_m).
\end{align*}
Here, we used the property that if $g \in G$ and $X \in \fg$, then $\exp(gXg^{-1}) = g\exp(X)g^{-1}$. 

Consequently, we find that
\begin{align*}
&d(e,\Psi_{m-1}(y_2,\ldots,y_m)^{-1}\exp(-x_1)\exp(y_1)\Psi_{m-1}(y_2,\ldots,y_m))
\\
&\leq
\left|\Psi_{m-1}(y_2,\ldots,y_m)^{-1}\log(\exp(-x_1)\exp(y_1))\Psi_{m-1}(y_2,\ldots,y_m)\right|
\end{align*}

Because $y_2,\ldots,y_m \in B(0,rm^{-1})$, in the same way as we obtained \eqref{eq:triangle_distance}, we find that
$$
|\Psi_{m-1}(y_2,\ldots,y_m)| \leq B\sum_{i=2}^m |y_i| \leq Br.
$$
Now, because Lie groups are complete as Riemannian manifold, the set $\overline{B(e,Br)} \subset G$ is compact. Combining everything and applying Lemma \ref{lemma:bounded_conjugation}, there exists a constant $C > 0$ such that
\begin{align*}
&\left|\Psi_{m-1}(y_2,\ldots,y_m)^{-1}\log(\exp(-x_1)\exp(y_1))\Psi_{m-1}(y_2,\ldots,y_m)\right|
\\
&\leq 
C|\log(\exp(-x_1)\exp(y_1))|.
\end{align*}

By Proposition \ref{prop:log_Lipschitz} there exists a (possibly different) constant $C > 0$ such that
$$
|\log(\exp(-x_1)\exp(y_1))| \leq C|x_1 - y_1|.
$$
Consequently, we get that there exists a constant $C > 0$ such that
$$
d(e,\Psi_{m-1}(y_2,\ldots,y_m)^{-1}\exp(-x_1)\exp(y_1)\Psi_{m-1}(y_2,\ldots,y_m)) \leq C|x_1 - y_1|,
$$
and hence,
\begin{align*}
&d(\Psi_m(x_1,\ldots,x_m),\Psi_m(y_1,\ldots,y_m))
\\
&\leq
d(\Psi_{m-1}(x_1,\ldots,x_m),\Psi_{m-1}(y_1,\ldots,y_m)) + C|x_1 - y_1|.
\end{align*}
Iterating this procedure, we find that
$$
d(\Psi_m(x_1,\ldots,x_m),\Psi_m(y_1,\ldots,y_m)) \leq C\sum_{i=1}^m |x_i - y_i|.
$$

It thus follows that if 
$$
(x_1,\ldots,x_m) \in B(y_1,(Cm)^{-1}\epsilon) \times \cdots \times B(y_m,(Cm)^{-1}\epsilon),
$$
then
$$
d(\Psi_m(x_1,\ldots,x_m),\Psi_m(y_1,\ldots,y_m)) < C\sum_{i=1}^m (Cm)^{-1}\epsilon = \epsilon,
$$
which proves the claim.

\end{proof}

We need one more result, which allows us to cut up a curve $\gamma \in \Aa\Cc$ in an appropriate way.

\begin{proposition}\label{prop:fundamental_theorem}
Let $\gamma \in \Aa\Cc([0,1];G)$ be arbitrary. Assume that $\dot\gamma \in L^\infty([0,1],\fg)$. Then for each $m$ large enough, the vectors 
$$
\log\left(\gamma\left(\frac {i-1}m\right)^{-1}\gamma\left(\frac im\right)\right) \in \fg
$$ 
are well-defined for $i = 1,\ldots, m$. Furthermore, there exist constants $L_m$ with $\lim_{m\to\infty} L_m = 0$ such that
$$
\left|\log\left(\gamma\left(\frac {i-1}m\right)^{-1}\gamma\left(\frac im\right)\right) - \int_{\frac{i-1}{m}}^{\frac im} \dot\gamma(t)\ud t\right| \leq L_m\frac1m||\dot\gamma||_\infty.
$$
\end{proposition}
\begin{proof}
First of all, because $\gamma$ is continuous and $[0,1]$ is compact, it is actually uniformly continuous. Consequently, we can take $m \in \NN$ large enough, so that for $i = 1,\ldots,m$ the vectors 
$$
\log\left(\gamma\left(\frac {i-1}m\right)^{-1}\gamma\left(\frac im\right)\right) \in \fg
$$ 
are well-defined.

Now consider the function $f:[\frac{i-1}{m},\frac im] \to \fg$ given by
$$
f(r) = \log\left(\gamma\left(\frac {i-1}m\right)^{-1}\gamma(r)\right).
$$
Then 
$$
f'(r) = \dd\log\left(\gamma\left(\frac {i-1}m\right)^{-1}\gamma(r)\right)(\dot\gamma(r)),
$$
where again we used the identification of $T_{\gamma(r)}G$ with $\fg$. Consequently, we have
\begin{align*}
\log\left(\gamma\left(\frac {i-1}m\right)^{-1}\gamma\left(\frac im\right)\right)
&= 
f\left(\frac{i}{m}\right) - f\left(\frac{i-1}{m}\right)
\\
&= 
\int_{\frac{i-1}{m}}^{\frac im} f'(r)\ud r
\\
&=  
\int_{\frac{i-1}{m}}^{\frac im} \dd\log\left(\gamma\left(\frac {i-1}m\right)^{-1}\gamma(r)\right)(\dot\gamma(r)) \ud r.
\end{align*}

With this expression at hand, we can estimate
\begin{align}
&\left|\log\left(\gamma\left(\frac {i-1}m\right)^{-1}\gamma\left(\frac im\right)\right) - \int_{\frac{i-1}{m}}^{\frac im} \dot\gamma(r)\ud t\right|
\\
&\leq
\int_{\frac{i-1}{m}}^{\frac im} \left|\left|\dd\log\left(\gamma\left(\frac {i-1}m\right)^{-1}\gamma(r)\right) - I\right|\right||\dot\gamma(r)| \ud r \label{eq:straigt_line}
\\
&\leq
||\dot\gamma||_\infty \int_{\frac{i-1}{m}}^{\frac im}\left|\left|\dd\log\left(\gamma\left(\frac {i-1}m\right)^{-1}\gamma(r)\right) - I\right|\right| \ud r
\end{align}

Now it follows from \eqref{eq:series_inverse} that (see also \cite[Chapter 5]{Hal15} or \cite[Chapter 2]{Vaj84})
$$
\dd\log\left(\gamma\left(\frac {i-1}m\right)^{-1}\gamma(r)\right) - I = \sum_{k=1}^\infty \frac{(-1)^{k+1}}{k(k+1)}\left(e^{\ad_{\log\left(\gamma\left(\frac {i-1}m\right)^{-1}\gamma(r)\right)}} - I\right)^k.
$$
From this it follows that
\begin{equation} \label{eq:estimate_dlog}
\left|\left|\dd\log\left(\gamma\left(\frac {i-1}m\right)^{-1}\gamma(r)\right) - I\right|\right| \leq \sum_{k=1}^\infty \frac{1}{k(k+1)}\left(e^{||\ad_{\log\left(\gamma\left(\frac {i-1}m\right)^{-1}\gamma(r)\right)}||} - 1\right)^k.
\end{equation}
Because $\gamma$ is uniformly continuous on $[0,1]$, together with the continuity of the logarithm, we have that 
$$
\lim_{m\to\infty} \sup_{1\leq i \leq m}\sup_{r \in [\frac{i-1}{m},\frac im]} \left|\log\left(\gamma\left(\frac {i-1}m\right)^{-1}\gamma(r)\right)\right| =0.
$$
But then also
$$
\lim_{m\to\infty} \sup_{1\leq i \leq m}\sup_{r \in [\frac{i-1}{m},\frac im]} \left|\left|\ad_{\log\left(\gamma\left(\frac {i-1}m\right)^{-1}\gamma(r)\right)}\right|\right| = 0,
$$
so that the upper bound in \eqref{eq:estimate_dlog} tends to 0 if $m$ goes to infinity. We can thus find constants $L_m$ with $\lim_{m\to\infty} L_m = 0$ such that
$$
\left|\left|\dd\log\left(\gamma\left(\frac {i-1}m\right)^{-1}\gamma(r)\right) - I\right|\right| \leq L_m
$$
for all $i = 1,\ldots, m$ and all $r \in [\frac{i-1}{m},\frac im]$. If we plug this into \eqref{eq:straigt_line}, we find
$$
\left|\log\left(\gamma\left(\frac {i-1}m\right)^{-1}\gamma\left(\frac im\right)\right) - \int_{\frac{i-1}{m}}^{\frac im} \dot\gamma(r)\ud t\right| \leq ||\dot\gamma||_\infty\int_{\frac{i-1}{m}}^{\frac im} L_m \ud r = L_m\frac1m||\dot\gamma||_\infty
$$
as desired.

\end{proof}

With the final preparations done, we can prove the lower bound of the large deviation principle for $\{\sigma_n^n\}_{n\geq1}$.

\begin{proposition}\label{prop:lower_bound_LDP}
Let the assumptions of Theorem \ref{theorem:Cramer_Lie} be satisfied. Then for every $U \subset G$ open we have
$$
\liminf_{n\to\infty} \frac1n\log\PP(\sigma_n^n \in U) \geq -\inf_{g\in U} I_G(g),
$$
where $I_G$ is the good rate function given by \eqref{eq:ratefunction}.
\end{proposition}

\begin{proof}
Let $U \subset G$ be open. Fix $g \in U$ and a curve $\gamma \in \Aa\Cc([0,1];G)$ with $\gamma(0) = e$ and $\gamma(1) = g$. We will show that 
$$
\liminf_{n\to\infty} \frac1n\log\PP\left(\left(\frac1n*\Ss\right)_n \in U\right) \geq -\int_0^1 \Lambda^*(\dot\gamma(t))\ud t.
$$

If $\int_0^1 \Lambda^*(\dot\gamma(t))\ud t = \infty$, the above is certainly true. Hence, we assume that $\int_0^1 \Lambda^*(\dot\gamma(t))\ud t < \infty$. Because $\Lambda$ is the log-moment generating function of a bounded random variable, it follows $\Lambda^*$ is finite only on a bounded set, referred to as its domain. Consequently, because $\int_0^1 \Lambda^*(\dot\gamma(t))\ud t < \infty$, it must be that $\dot\gamma(t)$ is in the domain of $\Lambda^*$ for almost all $t$. But then we have that $||\dot\gamma||_\infty < \infty$.  

By the same reasoning as in the proof of Proposition \ref{prop:fundamental_theorem}, we can take $m \in \NN$ large enough, so that we can define for $i = 1,\ldots,m$ the vectors 
$$
y_i^m := \log\left(\gamma\left(\frac {i-1}m\right)^{-1}\gamma\left(\frac im\right)\right) \in \fg.
$$ 
Let $\Psi_m:\fg^m\to G$ be again the map given by 
$$
\Psi_m(x_1,\ldots,x_m) = \exp(x_1)\cdots\exp(x_m),
$$
so that $g = \Psi_m(y_1^m,\ldots,y_m^m)$. 

Because $U$ is open, there exists an $\epsilon > 0$ such that $B(g,\epsilon) \subset U$. By Proposition \ref{prop:continuity_repeated_exponential}, for $m$ large enough, there exists a constant $C > 0$ independent of $m$, such that if 
$$
(x_1,\ldots,x_m) \in B(y_1^m,(Cm)^{-1}\epsilon) \times \cdots \times B(y_m^m,(Cm)^{-1}\epsilon),
$$
then $\Psi_m(x_1,\ldots,x_m) \in B(g,\epsilon)$. 

Now define for $i = 1,\ldots,m$ the vectors
$$
\tilde y_i^m := \int_{\frac{i-1}{m}}^{\frac im} \dot\gamma(t)\ud t.
$$

By Proposition \ref{prop:fundamental_theorem}, for $m$ large enough there exists a constant $L_m$ such that for $i = 1,\ldots,m$ we have
$$
|y_i^m - \tilde y_i^m| \leq L_m\frac1m||\dot\gamma||_\infty
$$
and $\lim_{m\to\infty} L_m = 0$. But then we have for $m$ large enough that $B(\tilde y_i^m, (2Cm)^{-1}\epsilon) \subset B(y_i^m,(Cm)^{-1}\epsilon)$. Consequently, we have that if 
$$
(x_1,\ldots,x_m) \in B(\tilde y_1^m,(2Cm)^{-1}\epsilon) \times \cdots \times B(\tilde y_m^m,(2Cm)^{-1}\epsilon),
$$
then $\Psi_m(x_1,\ldots,x_m) \in B(g,\epsilon)$

Now, let $\left(Y_{\lfloor m^{-1}n\rfloor}^{n,m,1},\ldots,Y_{\lfloor m^{-1}n\rfloor}^{n,m,m}\right)$ be again as in \eqref{eq:random_walk_vector}, so that 
$$
\Psi_m\left(\left(Y_{\lfloor m^{-1}n\rfloor}^{n,m,1},\ldots,Y_{\lfloor m^{-1}n\rfloor}^{n,m,m}\right)\right) = \sigma_n^n.
$$

Using the above, we have
\begin{align*}
\PP\left(\sigma^n_n \in U\right) 
&\geq 
\PP\left(\left(Y_{\lfloor m^{-1}n\rfloor}^{n,m,1},\ldots,Y_{\lfloor m^{-1}n\rfloor}^{n,m,m}\right) \in B(\tilde y_1^m,(2Cm)^{-1}\epsilon) \times \cdots \times B(\tilde y_m^m,(2Cm)^{-1}\epsilon)\right)
\\
&=
\prod_{i=1}^m \PP\left(Y_{\lfloor m^{-1}n\rfloor}^{n,m,i} \in B(\tilde y_i^m,(2Cm)^{-1}\epsilon)\right)
\\
&=
\prod_{i=1}^m \PP\left(Y_{\lfloor m^{-1}n\rfloor}^{n,m,1} \in B(\tilde y_i^m,(2Cm)^{-1}\epsilon)\right)
\end{align*}
Here we used again the fact that $Y_{\lfloor m^{-1}n\rfloor}^{n,m,1},\ldots,Y_{\lfloor m^{-1}n\rfloor}^{n,m,m}$ are i.i.d., which follows from the fact that the sequence $\{X_n\}_{n\geq1}$ is i.i.d., together with expression \eqref{eq:partition}.

Continuing, it follows from Proposition \ref{prop:replacement_sum} that 
$$
\left|Y_{\lfloor m^{-1}n\rfloor}^{n,m,1} - \frac1n\sum_{j=1}^{\lfloor m^{-1}n\rfloor}X_j\right| \leq C_m\frac1m,
$$ 
where $\lim_{m\to\infty} C_m = 0$. Consequently, for $m$ large enough, we have
$$
\left|Y_{\lfloor m^{-1}n\rfloor}^{n,m,1} - \frac1n\sum_{j=1}^{\lfloor m^{-1}n\rfloor}X_j\right| \leq (2Cm)^{-1}\frac\epsilon2.
$$

In that case we have
$$
\PP\left(Y_{\lfloor m^{-1}n\rfloor}^{n,m,1} \in B(\tilde y_i^m,(2Cm)^{-1}\epsilon)\right) \geq \PP\left(\frac1n\sum_{j=1}^{\lfloor m^{-1}n\rfloor}X_j \in B(\tilde y_i^m,(2Cm)^{-1}\epsilon/2)\right)
$$

By Cram\'er's theorem for vector spaces, it follows that $\{\frac1n\sum_{j=1}^{\lfloor m^{-1}n\rfloor}X_j\}_{n\geq 0}$ satisfies the large deviation principle in $\fg$ with good rate function $I_m(x) = \frac1m\Lambda^*(mx)$. Hence, we find that
\begin{align*}
\liminf_{n\to\infty} \frac1n\log\PP\left(\sigma^n_n \in U\right)
&\geq
\sum_{i=1}^m \liminf_{n\to\infty} \frac1n \log \PP\left(Y_{\lfloor m^{-1}n\rfloor}^{n,m,1} \in B(\tilde y_i^m,(2Cm)^{-1}\epsilon)\right)
\\
&\geq
\sum_{i=1}^m \liminf_{n\to\infty} \frac1n \log \PP\left(\frac1n\sum_{j=1}^{\lfloor m^{-1}n\rfloor}X_j \in B(\tilde y_i^m,(2Cm)^{-1}\epsilon/2)\right)
\\
&\geq
\sum_{i=1}^m -I_m(\tilde y_i^m)
\\
&=
-\frac1m\sum_{i=1}^m \Lambda^*(m\tilde y_i^m).
\end{align*}

We are done once we show that
$$
\frac1m\sum_{i=1}^m \Lambda^*(m\tilde y_i^m) \leq \int_0^1 \Lambda^*(\dot\gamma(t)) \ud t.
$$

By the convexity of $\Lambda^*$ and Jensen's inequality, we have
$$
\Lambda^*(m\tilde y_i^m) = \Lambda^*\left(m\int_{\frac{i-1}{m}}^{\frac im} \dot\gamma(t)\ud t\right) \leq m\int_{\frac{i-1}{m}}^{\frac im} \Lambda^*(\dot\gamma(t))\ud t.
$$
From this it follows that
$$
\frac1m\sum_{i=1}^m \Lambda^*(m\tilde y_i^m) \leq \sum_{i=1}^m \int_{\frac{i-1}{m}}^{\frac im} \Lambda^*(\dot\gamma(t))\ud t = \int_0^1 \Lambda^*(\dot\gamma(t))\ud t,
$$
which concludes the proof.

\end{proof}

\smallskip

\textbf{Acknowledgment}
RV was supported by the Peter Paul Peterich Foundation via the TU Delft University Fund.

\printbibliography

\end{document}